\documentclass{elsarticle}%
\usepackage{amsfonts}
\usepackage{amsthm}

\usepackage{amsmath}
\usepackage{amssymb}
\usepackage{bbold}
\usepackage{mathtools}
\usepackage{xcolor}
\usepackage{soul}
\usepackage{mathrsfs}
\usepackage{tabularx}
\usepackage[top=1in, bottom=1in, left=1in, right=1in]{geometry}
\usepackage{graphicx}%
\providecommand{\U}[1]{\protect\rule{.1in}{.1in}}

\newtheorem{theorem}{Theorem}[section]
\newtheorem{lemma}{Lemma}[section]
\newtheorem{corollary}{Corollary}[section]

\theoremstyle{definition}

\theoremstyle{remark}
\newtheorem{remark}{Remark}[section]
\newtheorem{example}{Example}[section]

\numberwithin{equation}{section}
\begin{document}
    \begin{frontmatter}

\title{Singular mean-field backward  stochastic  Volterra integral equations in infinite dimensional spaces
}



\author[1]{Javad A. Asadzade}
\ead{javad.asadzade@emu.edu.tr}
\author[1,2]{Nazim I. Mahmudov}
\ead{nazim.mahmudov@emu.edu.tr}

\address [1] {Department of Mathematics, Eastern Mediterranean University, Mersin 10, 99628, T.R. North Cyprus, Turkey}
\address [2] {Research Center of Econophysics, Azerbaijan State University of Economics (UNEC), Istiqlaliyyat Str. 6, Baku 1001,Azerbaijan}



\begin{abstract}
This paper investigates the well-posedness of singular mean-field backward stochastic Volterra integral equations (MF-BSVIEs) in infinite-dimensional spaces. We consider the equation: 

\[X(t) = \Psi(t) + \int_t^b P\big(t, s, X(s), \aleph(t, s), \aleph(s, t), \mathbb{E}[X(s)], \mathbb{E}[\aleph(t, s)], \mathbb{E}[\aleph(s, t)]\big) ds - \int_t^b \aleph(t, s) dB_s,
\]

where the focus lies on establishing the existence and uniqueness of adapted M-solutions under appropriate conditions. A key contribution of this work is the development of essential lemmas that provide a rigorous foundation for analyzing the well-posedness of these equations. In addition, we extend our analysis to singular mean-field forward stochastic Volterra integral equations (MF-FSVIEs) in infinite-dimensional spaces, demonstrating their solvability and unique adapted solutions. Finally, we strengthen our theoretical results by applying them to derive stochastic maximum principles, showcasing the practical relevance of the proposed framework. These findings contribute to the growing body of research on mean-field stochastic equations and their applications in control theory and mathematical finance.
\end{abstract}


\end{frontmatter}

\section{INTRODUCTION}
	
	\label{Sec:intro}
Backward Stochastic Differential Equations (BSDEs) and Backward Volterra Integral Equations (BVIEs) are two closely related classes of equations that have become fundamental tools in modern stochastic analysis, providing powerful frameworks for addressing complex problems in various domains. BSDEs, introduced by Jean-Michel Bismut  in 1973 \cite{Bismut} and later extended by Étienne Pardoux and Shige Peng in 1990 \cite{Pardoux}, model systems that evolve backward in time, starting from a known terminal condition and determining the system’s evolution backward to an initial time. These equations are particularly useful in areas such as optimal control, mathematical finance, and the nonlinear Feynman-Kac formula \cite{Ma}, where the future state is specified and the dynamics leading to it need to be understood.

While BSDEs capture the backward evolution of systems, BVIEs extend this concept by incorporating integral terms that reflect the memory or history of the system. This inclusion allows BVIEs to model systems where the future state is influenced not only by the current state but also by the entire past trajectory, making them invaluable in applications involving delayed feedback, memory effects, and time-dependent processes. The connection between BSDEs and BVIEs enables the modeling of more complex systems, particularly in stochastic control problems and mathematical finance, where the path of a process up to the present moment significantly influences its future outcomes.
Numerous researchers have significantly contributed to the theory of backward stochastic Volterra integral equations (BSVIEs). Notable contributors include Lin, Yong, Anh, Grecksch, Wang, Zheng, and Hamaguchi, among others (for further details, see \cite{Ahmadova2, Anh, Gilding, Hamaguchi, Hamaguchi2, Lin, Mahmudov, Mahmudov2, Peng, Wang, Wang2, Yong, Yong2}).

Yong \cite{Yong2} introduced BSVIEs in the subsequent form:  
\begin{align}\label{p1}
X(t) = \rho(t) - \int_t^b h(t, s, X(s), \aleph(t, s), \aleph(s, t)) ds - \int_t^b \aleph(t, s)  dB_s, \quad t \in [0, b].  
\end{align}  
This representation exhibits unique features, including the dependence of \(\aleph(t, s)\) on \(t\) and the drift term's interaction with both \(\aleph(t, s)\) and \(\aleph(s, t)\). Later, in \cite{Yong}, Yong explored the well-posedness and regularity of adapted M-solutions for BSVIEs. It's worth noting that the results in \cite{Yong} differ from those in the present work. For example, Theorem 3.7 in \cite{Yong} is based on different assumptions for \(f\) and \(g\), where the Lipschitz coefficients depend on \(t\) and \(s\). In contrast, this work assumes these coefficients are positive constants.

More recently, Wang and Yong \cite{Wang2} studied BSVIEs and demonstrated how a specific case could be simplified to the subsequent form:  
\begin{align}\label{p2}
X(t) = \xi + \int_t^b g(s, X(s), \aleph(t, s)) ds - \int_t^b \aleph(t, s)  dB_s, \quad t \in [0, b].  
\end{align}  
Here, the solution \((X(\cdot), \aleph(\cdot, \cdot))\) consists of stochastic processes. Equation \eqref{p2} can be compared to the integral formulation of a backward stochastic differential equation (BSDE):  
\begin{align}\label{p3}
X(t) = \xi + \int_t^b g(s, X(s), \aleph(s)) ds - \int_t^b \aleph(s)  dB_s, \quad t \in [0, b].  
\end{align}  
If Equation \eqref{p2} has a unique \(\mathcal{F}_t\)-adapted solution \((X(\cdot), \aleph(\cdot))\), this solution also satisfies Equation \eqref{p1}, where \(\aleph(t, s) = \aleph(s)\). Thus, BSVIEs naturally extend the classical framework of BSDEs.

A crucial extension of both BSDEs and BVIEs is the concept of mean field interactions. Mean field theory is important because it captures the behavior of large systems composed of many interacting components, where each component’s behavior is influenced by the average state of the system as a whole. In stochastic systems, this leads to models where the evolution of an individual unit is influenced not only by its own state but also by the collective influence of all other units in the system. This is particularly relevant in fields such as population dynamics, large-scale financial markets, and social systems, where interactions between agents cannot be neglected.

The introduction of mean field interactions into BSDEs and BVIEs allows for more realistic and scalable models that account for the collective behavior of many interacting entities. In optimal control problems and finance, for example, mean field game theory provides insights into the optimal strategies of individual agents who are aware of and influenced by the average behavior of others. This approach leads to a deeper understanding of equilibrium dynamics in systems with a large number of agents, enriching the analysis and offering solutions that would be difficult to achieve with individual agent-based models alone.

Together, the integration of BSDEs, BVIEs, and mean field theory provides a comprehensive framework for solving nonlinear problems and optimizing decision-making strategies under uncertainty, particularly in complex, large-scale systems. Their applications span a wide range of fields, from financial portfolio optimization and derivative pricing to the design of control systems and the modeling of dynamical processes in uncertain environments. The combination of these tools is essential for addressing real-world challenges where past, present, and future events are intricately connected, and where the collective behavior of many interacting components plays a significant role in shaping the system's evolution.

The major contributions to the theory of MF-BSVIEs and MF-BSDEs have been made by researchers such as Buckdahn, R., Djehiche, B., Li, J., Peng, S., Agram, N., Hu, Y., Øksendal, B., Andersson, D., and others (for additional references, see \cite{Ahmadova,  Agram, Andersson, Buckdahn,Buckdahn2, Jaber, Li, Shi, Shi1,  Wang3, Wang4,  Wu, Yang, Zhu}). 

For instance, in \cite{Buckdahn}, the authors explore a specific mean-field problem using a purely stochastic framework. They analyze the solution \((Y, Z)\) of a mean-field BSDE driven by a forward stochastic differential equation (SDE) of McKean–Vlasov type, with solution \(X\). A key focus of their work is the approximation of this system by the solutions \((X^N, Y^N, Z^N)\) of decoupled forward-backward equations, where the coefficients depend on \(N\) independent copies of \((X^N, Y^N, Z^N)\). They establish that the rate of convergence for this approximation is of order \(1/\sqrt{N}\).
Moreover, their innovative approach to approximation enables a detailed characterization of the asymptotic behavior of \(\sqrt{N}(X^N - X, Y^N - Y, Z^N - Z)\). They prove that this triplet converges in distribution to the solution of a forward-backward stochastic differential equation of mean-field type. This limiting equation is influenced not only by a Brownian motion but also by an independent Gaussian field, highlighting the intricate interplay of randomness in such systems.

On the other hand, in \cite{Buckdahn2}, the authors establish several fundamental results for mean-field BSDEs. They prove the existence and uniqueness of solutions, along with the comparison and converse comparison theorems. They also analyze McKean–Vlasov SDEs and investigate decoupled mean-field Forward–Backward SDEs (FBSDEs). 
The associated value function \(u\) is shown to be a deterministic function that is Lipschitz continuous with respect to \(x\) and \(1/2\)-Hölder continuous in \(t\). Moreover, \(u\) satisfies the Dynamic Programming Principle (DPP), a critical property proved using Peng’s backward semigroups \cite{Peng}. A modified definition of these semigroups is introduced, simplifying the argument for showing that \(u\) is a viscosity solution of the associated partial differential equation (PDE). Finally, the uniqueness of the viscosity solution is demonstrated within the space of continuous functions with polynomial growth.
\bigskip

Recently, latest main contributions for singular BSVIEs in infinite-dimensional spaces was given by Wang and Zeng \cite{Wang}. They introduce the concept of singular BSVIEs and establish their well-posedness by addressing various singularity conditions. These conditions encompass a wide range of kernels, including fractional kernels, Volterra Heston model kernels, and completely monotone kernels. They also provide fresh perspectives on forward stochastic Volterra integral equations. Inspired by challenges in mathematical physics, such as the viscoelasticity and thermoviscoelasticity of materials, as well as heat conduction in materials with memory, these insights find applications in optimal control problems. 
Such problems extend to abstract stochastic Volterra integral equations, fractional stochastic evolution equations, and stochastic evolutionary integral equations. Furthermore, the framework of BSVIEs proves instrumental in addressing the maximum principle for controlled stochastic delay evolution equations. 
One notable advantage of this framework is its ability to naturally incorporate past states into the cost functional. This capability offers a more flexible and comprehensive approach to managing control problems that are influenced by memory effects, thus enriching the analysis and solutions in these complex systems.
\bigskip

Inspired by this remarkable result, we address the question in this article of whether \textbf{it is possible to obtain an analogue of these results for the mean-field}, which is the key focus mentioned earlier. To explore this, we consider the subsequent singular MF-BSVIEs \eqref{r0}. Such that,  we fix \( b > 0 \) and consider two separable Hilbert spaces, \( H \) and \( V \). Let \( L_2^0 \coloneqq L_2(V; H) \). We work on a complete filtered probability space \( (\Omega, \mathcal{F}, \{\mathcal{F}_t\}_{t \in [0,b]}, \mathbb{P}) \), where a \( V \)-valued cylindrical Brownian motion \( B(\cdot) \) is defined. The filtration \( \{\mathcal{F}_t\} \) is the natural filtration of \( B(\cdot) \), augmented with all \( \mathbb{P} \)-null sets in \( \mathcal{F} \).

	\begin{align}\label{r0}
X(t)&=\Psi(t)+\int_{t}^{b}P(t,s,X(s),\aleph(t,s),\aleph(s,t),\mathbb{E}\big[X(s)\big], \mathbb{E}\big[\aleph(t,s)\big], \mathbb{E}\big[\aleph(s,t)\big])ds-\int_{t}^{b}\aleph(t,s)dB_{s}.
	\end{align}
where the mappings \( \Psi(\cdot) \) and \( P(\cdot) \), called the free term and generator of equation \eqref{r0}, are given functions valued in \( H \), satisfying certain singular assumptions to be outlined later. Our objective is to determine an adapted pair of processes \( (X(\cdot), \aleph(\cdot, \cdot)) \) that fulfill equation \eqref{r0} in the conventional Itô framework.

\bigskip
    To begin with, let us examine the subsequent significant example, which will play a crucial role in deriving key results and insights in the forthcoming sections.
    
\begin{example}
Consider the subsequent Caputo fractional MF-BSDEs of order \(\gamma \in \left(\frac{1}{2}, 1\right)\), defined on the interval \([0, b]\):

\begin{equation}\label{eq1}
\begin{cases}
\big({^{C}_{t}}D^{\gamma}_{b} x\big)(t) = Ax(t) -  \rho\big(t, x(t), z(t, s), z(s, t), \mathbb{E}[x(t)], \mathbb{E}[z(t, s)], \mathbb{E}[z(s, t)]\big) - z(t, s)\frac{dB_{t}}{dt}, \\
x(b) = \xi,
\end{cases}
\end{equation}

where
\begin{equation}
{^{C}_{t}}D^{\gamma}_{b} x(t) \triangleq -\frac{1}{\Gamma(1-\gamma)} \int_{t}^{b} \frac{x'(s)}{(s-t)^{\gamma}} ds.
\end{equation}

Here, \(A\) is a constant matrix, and \(B(\cdot)\) represents a standard \(m\)-dimensional Brownian motion on a probability space \((\Omega, \mathcal{F}, \{\mathcal{F}_t\}_{t \geq 0}, \mathbb{P})\), where \(\{\mathcal{F}_t\}\) is a filtration satisfying the usual conditions. The term \(\frac{dB}{dt}\) is interpreted as white noise, representing the generalized derivative of Brownian motion. Additionally:
\begin{itemize}
    \item \(\xi \in L^2(\mathcal{F}_b; \mathbb{R}^n)\),
    \item \(\rho : [0, b] \times \mathbb{R}^n \times \mathbb{R}^{n \times m} \times \mathbb{R}^n \times \mathbb{R}^{n \times m} \to \mathbb{R}^n\) is a measurable function satisfying \(\rho(\cdot, 0, 0, 0, 0) \in L^2(0, b; \mathbb{R}^n)\).
\end{itemize}

For all \(x_1, x_2, \bar{x}_1, \bar{x}_2 \in \mathbb{R}^n\), \(z_1, z_2, \bar{z}_1, \bar{z}_2 \in \mathbb{R}^{n \times m}\), and \(t \in [0, b]\), there exists a constant \(L > 0\) such that:
\begin{align*}
&\quad \big\vert \rho(t, x_1, z_1, \xi_1, \bar{x}_1, \bar{z}_1, \bar{\xi}_1) -  \rho(t, x_2, z_2, \xi_2, \bar{x}_2, \bar{z}_2, \bar{\xi}_2) \big\vert^2 \\
& \leq L \Big( \vert x_1 - x_2 \vert^2 + \vert z_1 - z_2 \vert^2 + \vert \xi_1 - \xi_2 \vert^2 + \vert \bar{x}_1 - \bar{x}_2 \vert^2 + \vert \bar{z}_1 - \bar{z}_2 \vert^2 + \vert \bar{\xi}_1 - \bar{\xi}_2 \vert^2 \Big).
\end{align*}

A mild solution of the  \eqref{eq1} is given by:
\begin{equation}\label{r1}
x(t) = \xi + \frac{1}{\Gamma(\gamma)} \int_{t}^{b} \frac{\tilde{\rho}(s) - Ax(s)}{(s-t)^{1-\gamma}} ds - \frac{1}{\Gamma(\gamma)} \int_{t}^{b} \frac{z(t, s)}{(s-t)^{1-\gamma}} dB_s,
\end{equation}
where:
\[
\tilde{\rho}(s) = f\big(s, x(s), z(t, s), z(s, t), \mathbb{E}[x(s)], \mathbb{E}[z(t, s)]\big).
\]

By defining new variables:
\begin{align*}
\begin{cases}
X(t) \triangleq x(t), \quad \aleph(t, s) \triangleq \frac{1}{\Gamma(\gamma)} (s-t)^{\gamma-1} z(t, s), \quad \aleph(s, t) \triangleq z(s, t), \\
\widehat{\rho}(s) = \rho\big(s, X(s), \frac{\Gamma(\gamma)}{(s-t)^{\gamma-1}} \aleph(t, s), \aleph(s, t), \mathbb{E}[X(s)], \mathbb{E}[\frac{\Gamma(\gamma)}{(s-t)^{\gamma-1}} \aleph(t, s)], \mathbb{E}[\aleph(s, t)]\big), \\
\widehat{P}(t, s) = \frac{1}{\Gamma(\gamma)} (s-t)^{\gamma-1} \big[\widehat{\rho}(s) - AX(s)\big],
\end{cases}
\end{align*}
the mild solution\eqref{r1} can be expressed as a special case of a BSVIE with:
\begin{align}\label{r2}
\begin{cases}
\Psi(t) \triangleq \xi, \\
\widehat{P}(t, s) \triangleq \frac{1}{\Gamma(\gamma)} (s-t)^{\gamma-1} \big[\widehat{\rho}(s) - AX(s)\big].
\end{cases}
\end{align}
\end{example}

The example provided above expresses the singular part of the article, while now we will look at the subsequent example to clarify the infinite-dimensional part of the article. Thus, the motivation for the research in the article comes from mathematical physics. To achieve this, we introduce the subsequent example of MF-BSVIEs \eqref{r0}, which are relevant to stochastic evolutionary integral equations that model phenomena such as viscoelasticity, thermoviscoelasticity in materials, incompressible fluids, and heat conduction in materials with memory. 

\begin{example}
We initiate our discussion with the framework of forward semilinear stochastic evolutionary integral equations. Consider a Hilbert space \( H \), and let \( A \) denote a densely defined, closed, and unbounded linear operator on \( H \) with domain \( D(A) \). Furthermore, let \( c \in L^1(0, b; \mathbb{R}^+) \) be a scalar-valued kernel function, representing a temporally dependent coefficient integral to the dynamics of the system. Consider
\begin{align}\label{30}
    Y(t)=y_{0}-\int_{0}^{t}c(t-s)AY(s)ds+\int_{0}^{t}\vartheta(s,Y(s), \mathbb{E}Y(s))ds+\int_{0}^{t}\Lambda(s,Y(s), \mathbb{E}Y(s))dB_{s},\, t\in[0,b],
\end{align}
where \( \vartheta : [0, b] \times H\times H \to H \) and \( \Lambda : [0, b] \times H \times H\to L_0 \) are Lipschitz continuous with linear growth. As will be demonstrated later, with different forms of \( c(\cdot) \), such equations find interesting applications in mathematical physics.

To analyze the well-posedness of \eqref{30}, we adopt the notion of a resolvent as introduced in \cite{Prüss}. A family of bounded linear operators \( \{T(t)\}_{t \geq 0} \) on the Hilbert space \( H \) is termed a resolvent for \eqref{30} if it meets the subsequent criteria: the mapping \( t \mapsto T(t) \) is strongly continuous, commutes with the operator \( A \), satisfies the initial condition \( T(0) = I \), and adheres to the resolvent equation given below:
\begin{align}\label{31}
    T(t)y=y-\int_{0}^{t}c(t-s)AT(s)yds,\, y\in D(A),\, t\in [0,b].
\end{align}
For a comprehensive treatment of the resolvent concept, we direct the reader to [\cite{Prüss}, Theorem 3.1, 3.2, 4.2, 4.3, and 4.4 in Chapter 5]. Once the bounded resolvent is constructed, a (mild) solution of \eqref{30} is characterized as a function \( Y(\cdot) \) that fulfills the associated stochastic Volterra integral equation outlined below:
\begin{align}\label{32}
    Y(t)=T(t)y_{0}+\int_{0}^{t}T(t-s)\vartheta(s,Y(s), \mathbb{E}Y(s))ds+\int_{0}^{t}T(t-s)\Lambda(s,Y(s), \mathbb{E}Y(s))dB_{s}
\end{align}
It can be considered a special case of \eqref{4.1} with:
\begin{align}\label{33}
    \begin{cases}
        \varphi(t)=T(t)y_{0},\\
        \Phi(s,Y(s), \mathbb{E}Y(s))=  T(t-s)\vartheta(s,Y(s), \mathbb{E}Y(s)),\\
        \Psi(s,Y(s), \mathbb{E}Y(s))=  T(t-s)\Lambda(s,Y(s), \mathbb{E}Y(s))
    \end{cases}
\end{align}
As indicated in \cite{Shi} and \cite{Yong}, when studying the optimal control problems for \eqref{33}, the subsequent kind 
of linear BSVIE is required:
\begin{align}\label{34}
    Y(t)&=\varphi(t)+\int_{t}^{b}[N^{*}_{1}(t)T^*(s-t)Y(s)+N^{*}_{2}(t)T^*(s-t)\aleph(s,t)]ds\nonumber\\
    &+\mathbb{E}\int_{t}^{b}[N^{*}_{3}(t)T^*(s-t)Y(s)+N^{*}_{4}(t)T^*(s-t)\aleph(s,t)]ds\\
    &-\int_{t}^{b} \aleph(t,s)dB_{s},\, t\in [0,b],\nonumber
\end{align}
with appropriate operator-valued functions \( N_1 , N_2, N_{3}\) and \( N_4 \). Clearly, equation (1.7) can be seen as a special case of \eqref{r0}.
\end{example}

\bigskip
The above discussion highlights the significance and intrigue of studying singular MF-BSVIEs in infinite-dimensional spaces, both from theoretical and applied perspectives.

\bigskip

The structure of the article is organized as follows to ensure a comprehensive exploration of the topic: 
\bigskip

 Section 2 introduces important foundational concepts and preliminary results. These are essential for understanding the main analysis later in the article.  In Section 3, we focus on the well-posedness of singular MF-BSVIEs. We explain the analytical challenges caused by their singular nature and describe the conditions needed for solutions to exist and be unique.  Next, Section 4 discusses singular MF-FSVIEs. It provides a similar analysis to Section 3 and highlights the connection between forward and backward systems.  
Finally, Section 5 explores how these theories can be applied to optimal control problems. We specifically consider cases with convex control regions. This section also explains how to set up the control problems and derives key conditions for finding optimal solutions.
 
\section{Mathematical backgraound}
 In this section, we present fundamental concepts that are crucial for the subsequent parts of the article, along with the functional spaces that form the central focus of the discussion.
 \bigskip
 
Let \( V \) and \( H \) be two separable Hilbert spaces. The space of all bounded linear operators mapping \( V \) to \( H \) is denoted by \( L(V; H) \), while \( L_2^0 \) represents the space of Hilbert-Schmidt operators from \( V \) to \( H \). Specifically, 

\[
L_2^0 \coloneqq \{T_{1} \in L(V; H) \mid \sum_{i=1}^\infty |T_{1}e_i|_H^2 < \infty \},
\]

where \( \{e_i\}_{i=1}^\infty \) is an orthonormal basis of \( V \). It can be demonstrated that \( L_2^0 \), when endowed with the inner product 

\[
\langle T_{1}, T_{2} \rangle_{L_2^0} \coloneqq \sum_{i=1}^\infty \langle T_{1}e_i, T_{2}e_i \rangle_H \quad \forall T_{1}, T_{2} \in L_2^0,
\]

forms a separable Hilbert space.
\bigskip

Next, we define triangular domains that will play a crucial role in our analysis. The first triangular domain is given by

\[
\triangle \coloneqq \{(t, s) \in [0, b]^2 \mid 0 \leq s < t \leq b \},
\]

which consists of all points \((t, s)\) in the square \([0, b]^2\) where the second coordinate \(s\) is strictly less than the first coordinate \(t\), and both \(t\) and \(s\) lie within \([0, b]\). For any fixed \(r \in [0, b)\), a restricted version of this domain is defined as 

\[
\triangle_{[r, b]} \coloneqq \{(t, s) \in [r, b]^2 \mid 0 \leq r \leq s < t \leq b \},
\]

which represents the subset of \(\triangle\) where both coordinates \(t\) and \(s\) are further constrained to lie in the interval \([r, b]\) with \(r \leq s < t\).

Similarly, we introduce a second triangular domain, denoted by 

\[
\triangle^* \coloneqq \{(t, s) \in [0, b]^2 \mid 0 \leq t < s \leq b \},
\]

which includes all points \((t, s)\) in \([0, b]^2\) where \(t\) is strictly less than \(s\), and both values fall within the interval \([0, b]\). For a fixed \(r \in [0, b)\), the corresponding restricted domain is defined as 

\[
\triangle^*_{[r, b]} \coloneqq \{(t, s) \in [r, b]^2 \mid 0 \leq r \leq t < s \leq b \}.
\]

This domain consists of pairs \((t, s)\) where \(t < s\) and both coordinates are within \([r, b]\), with the additional condition that \(r \leq t\). These triangular regions help delineate the relationships between the two variables \(t\) and \(s\) and provide structured domains for integration or other operations.
\bigskip

We begin by defining the space of \( H \)-valued square-integrable random variables as 

\[
L^2_{\mathcal{F}_t}(\Omega; H) \coloneqq \{\xi: \Omega \to H \mid \xi \text{ is } \mathcal{F}_t\text{-measurable, } \|\xi\|_2 \coloneqq \big(\mathbb{E}|\xi|_H^2\big)^{1/2} < \infty\}.
\]

This space consists of random variables taking values in \( H \), measurable with respect to the \( \sigma \)-algebra \( \mathcal{F}_t \), and having finite \( H \)-norm in expectation. It is evident that \( L^2_{\mathcal{F}_t}(\Omega; H) \) is a Banach space when equipped with the norm \( \|\cdot\|_2 \).

Next, we turn our attention to defining the functional spaces related to stochastic processes. Unless stated otherwise, we assume all stochastic processes \( \psi(t, \omega) \) to be at least measurable with respect to the product \( \sigma \)-algebra \( \mathcal{B}([0, b]) \otimes \mathcal{F}_b \), where \( \mathcal{B}([0, b]) \) represents the Borel \( \sigma \)-algebra on \([0, b]\). This ensures a sufficient level of regularity for the processes under consideration.

For any pair \( 0 \leq r \leq \delta \leq b \), we construct and impose the corresponding stochastic process spaces, tailored to specific intervals and aligned with the structure of the problem. These spaces provide a rigorous framework for analyzing the behavior and properties of stochastic processes across defined temporal domains.
\bigskip
\begin{align*}
    L^{2}_{F_{\delta}} (r,\delta; H) = \Big\{\psi  : [r, \delta] \times \Omega\to H \, :\, \psi \in \mathcal{B}([r, \delta]) \otimes \mathcal{F}_{\delta}-measurable\, and\, 
     \mathbb{E} \int_r^\delta |\psi (t)|_H^2 \, dt < \infty.\Big\}
\end{align*}
\begin{align*}
    L^{2}_{F} (r,\delta; H) =\Big\{ \psi \in  L^{2}_{F_{\delta}} (r,\delta; H) : \,\psi \, is\, a\, F-adapted. \Big\}
\end{align*}
\begin{align*}
     L^2_F(r,\delta; C([r,\delta]; H))=\Big\{\psi : [r, \delta] \times \Omega\to H \, :  \,\psi \, is\, a\, F-adapted, continuous\, and \, \mathbb{E}\Big( \max_{r\leq t\leq \delta} |\psi (t)|_H^2 \Big)<+\infty.\Big\}
\end{align*}
\begin{align*}
    L^2_F(r,\delta; L_0^2)=\Big\{ \psi  : [r, \delta] \times \Omega\to L_0^2 \, :  \, \psi (\cdot)\, is\, a \, F-adapted\, and\, \mathbb{E} \int_r^\delta \vert\psi (t)\vert_{L^0_2} \, dt < \infty. \Big\}
\end{align*}
      \begin{align*}
           L^2_F(r,\delta; L^2_F(\delta, b; L_0^2)) =\Big\{\zeta : [r, \delta] \times [\delta, b] \times \Omega\to L_0^2 \, :  \,\forall \tau \in [r, \delta] \,\zeta(\tau, \cdot) \in L^2_F(\delta, b; L_0^2)\, and\,  
     \mathbb{E} \int_r^\delta \int_\delta^b |\zeta(\tau,s)|_{L_0^2}^2 \, ds \, d\tau < \infty.\Big\}
      \end{align*}
      
     \begin{align*}
          L^2_F(r,\delta; L^2_F(r,\delta; L_0^2))=\Big\{   \zeta : [r, \delta]^2 \times \Omega\to L_0^2 \, :  \forall \tau \in [r, \delta], \,\zeta(\tau, \cdot) \in L^2_F(r, \delta ; L_0^2) \, and  \,
     \mathbb{E} \int_r^\delta \int_r^\delta \vert\zeta(\tau,s)\vert_{L_0^2}^2 \, ds \, d\tau < \infty.\Big\}
     \end{align*}
     
        \begin{align*}
        H^2[r, \delta] =\Big\{L^2_F(r, \delta; H) \times L^2_F(r, \delta; L^2_F(r, \delta; L_0^2))\Big\}
        \end{align*}

\bigskip

$\bullet$ \( L^2(\Delta^*; R^+) \) space consists of measurable functions \( f : \Delta^* \to \mathbb{R}^+ \), with the condition:
     \[
     \int_0^b \int_t^b |\rho(t,s)|^2 \, ds \, dt < \infty.
     \]
     This represents square-integrable functions over the interval \( [0, b] \) .
\bigskip

$\bullet$  \( L^2(\Delta; R^+) \) space consists of measurable functions \( f : \Delta \to \mathbb{R}^+ \), with the condition:
     \[
     \int_0^b \int_s^b |\rho(t,s)|^2 \, dt \, ds < \infty.
     \]
     This represents square-integrable functions with the roles of \( t \) and \( s \) reversed in the integration.
\bigskip

The set \( \mathscr{L}^2(\Delta^*; R^+) \) consists of functions \( f \in L^2(\Delta^*; R^+) \) that meet two conditions: First, for all \( t \in (0, b) \), the subsequent holds:

\[
\text{ess sup}_{t \in (0, b)} \left( \int_t^b |\rho(t, s)|^2 \, ds \right)^{1/2} < \infty.
\]

$\forall$ \( \varepsilon > 0 \), there is a finite partition \( \{ b_i \}_{i=0}^m \) of the interval \( (0, b) \), with \( 0 = b_0 < b_1 < \cdots < b_m = b \), such that for each \( i \in \{ 0, 1, \dots, m-1 \} \), the subsequent condition holds:

\[
\text{ess sup}_{t \in b_i, b_{i+1}} \left( \int_{b_i+1}^{t} | \rho(t, s)|^2 \, ds \right)^{1/2} < \varepsilon.
\]

The space \( \mathscr{L}^2( \Delta; R^+) \) can be defined in a similar way. To simplify, we omit the range space \( R^+ \) from the notation, resulting in:

\[
L^2(\Delta^*) \equiv L^2(\Delta^*; R^+), \quad L^2(\Delta) \equiv L^2(\Delta; R^+),
\]

and:

\[
\mathscr{L}^2(\Delta^*) \equiv \mathscr{L}^2(\Delta^*; R^+), \quad \mathscr{L}^2(\Delta) \equiv \mathscr{L}^2(\Delta; R^+).
\]
\section{Well-posedness of singular MF-BSVIEs}

In this paragraph, we examine the well-posedness of the MF-BSVIE \eqref{eq1}. To make it easier for readers to follow along, we will first restate the equation for clarity. By doing so, we aim to establish a clear and logical foundation for analyzing its properties and behavior. The discussion will focus on demonstrating the conditions under which the equation has a unique and well-defined solution, thereby reinforcing its theoretical validity and practical relevance.

	\begin{align}\label{r3}
X(t)&=\Psi(t)+\int_{t}^{b}\widehat{F}(t,s)ds-\int_{t}^{b}\aleph(t,s)dB_{s}\quad t\in [0,b].
	\end{align}

  where the free term \(\Psi(\cdot) \in L^{2}_{F_{b}}(0, b; H)\) and the generator \(\widehat{F}\) are specified. A pair \((X(\cdot), \aleph(\cdot, \cdot)) \in H^2[0, b]\)  is considered an adapted \(M\)-solution of the BSVIE \eqref{r3} if the equation \eqref{r3} is held in the standard Ito sense for almost every \(t \in [0, b]\), and the subsequent condition is satisfies:
 \begin{align*}
X(t)=\mathbb{E}X(s)+\int_{0}^{t}\aleph(t,s)dB_{s},\quad a.e.\quad t\in [0,b].
 \end{align*}
We present the subsequent set of assumptions.
\bigskip

$(A_{1})$ 
Let \(F : \Delta^* \times H \times L^2_0 \times L^2_0 \times\times H \times L^2_0 \times L^2_0 \times \mathbb{R} \to H\) be a measurable function such that for all \((t,x_{1},z_{1},\xi_{1},x_{2},z_{2},\xi_{2}) \in [0, b] \times H \times L^2_0 \times L^2_0\times H \times L^2_0 \times L^2_0\), the mapping \(s \mapsto P(t,s,x_{1},z_{1},\xi_{1},x_{2},z_{2},\xi_{2})\) is \(F\)-progressively measurable, and \(P(t,s, 0, 0, 0,0,0,0) = 0\). Additionally, it satisfies that

\begin{align}
&\quad \vert P(t,s,x_{1},z_{1},\xi_{1},x_{2},z_{2},\xi_{2})-P(t,s,\bar{x}_{1},\bar{z}_{1},\bar{\xi}_{1},\bar{x}_{2},\bar{z}_{2},\bar{\xi}_{2})\vert\nonumber\\
&\leq L_{x_{1}}(t,s)\vert x_{1}-\bar{x}_{1}\vert_{H}+ L_{z_{1}}(t,s)\vert z_{1}-\bar{z}_{1}\vert_{L^{0}_{2}}+L_{\xi_{1}}(t,s)\vert \xi_{1}-\bar{\xi}_{1}\vert_{L^{0}_{2}}\\
&+ L_{x_{2}}(t,s)\vert x_{2}-\bar{y}_{2}\vert_{H}+ L_{z_{2}}(t,s)\vert z_{2}-\bar{z}_{2}\vert_{L^{0}_{2}}+L_{\xi_{2}}(t,s)\vert \xi_{2}-\bar{\xi}_{2}\vert_{L^{0}_{2}}.\nonumber
\end{align}

where \(L_{x_{1}},L_{x_{2}} \in L^2(\Delta^*)\),\quad  \(L_{\xi_1},L_{\xi_2} \in  \mathscr{L}^{2}(\Delta^{*})\), and \quad
\begin{align*}
\sup_{t \in (0,b)} \int_t^{b} L_{z_1}(t,s)^2 ds < \infty,\quad \sup_{t \in (0,b)} \int_t^{b} L_{z_2}(t,s)^2 ds < \infty.
\end{align*}

\begin{remark}
    In both Anh et al. \cite{Anh} and Yong \cite{Yong}, the authors established Lipschitz conditions analogous to those presented in $(A_{1})$ and posited that all the coefficients simultaneously satisfy the subsequent condition:

\begin{align}\label{w1}
\sup_{t \in [0,b]} \int_{t}^{b}  \rho(t,s)^{2 + \varepsilon} ds < \infty,
\end{align}

where \(f\) represents the functions \(L_{x_{1}},L_{x_{1}}\), \(L_{z_1},L_{z_2}\), and \(L_{\xi_{1}},L_{\xi_2}\) for some \(\varepsilon > 0\).
\end{remark}

\begin{remark}
    The scenario where \(P(t,s,0,0,0,0,0,0) \not= 0\) can be addressed as follows. We identify a new free term
    
    \begin{align*}
        \tilde{\Psi}(t)\overset{\triangle}{=} \Psi(t) + \int_{t}^{b} P(t,s,0,0,0,0,0,0) ds
    \end{align*}
    
    for \(t \in [0,b]\). We can demonstrate that \(\tilde{\Psi}(\cdot) \in L^2_{F_{b}}(0,b;H)\) if the subsequent condition holds:

\[
\mathbb{E} \left[ \int_{0}^{b} \left( \int_{t}^{b} |P(t,s,0,0,0,0,0,0)|_H ds \right)^2 dt \right] < \infty.
\]
\end{remark}

We introduce some lemmas that will be beneficial later. For every \( r, \delta \in [0, b) \), focus on the subsequent \( H \)-valued mean-field stochastic integral equation:

\begin{align}\label{f3.4}
\lambda(t, \tau) &= \Psi(t) + \int_{\tau}^{b} \eta(t, s, \mu(t, s), \mathbb{E}[\mu(t, s)]) \, ds- \int_{\tau}^{b} \mu(t, s) \, dB_{s}, \quad t \in [\delta, b], \; \tau \in [r, b],
\end{align}

where the function \( \eta : [\delta, b] \times [r, b] \times L^2_0 \times L^2_0 \to H \) is given, with \( \eta(t, s, \mu, \mathbb{E}[\mu]) \) representing the generator term that depends on both the process \( \mu(t, s) \) and its mean field component \( \mathbb{E}[\mu(t, s)] \).
The unknown process is \( (\lambda(\cdot, \cdot), \mu(\cdot, \cdot)) \), where \( (\lambda(t, \cdot), \mu(t, \cdot)) \) is \( \mathcal{F} \)-adapted  \( \forall\, t \in [\delta, b] \).
\bigskip

This equation can be regarded as:
\bigskip

$\bullet$ A family of infinite-dimensional MF-BSDEs   on \( [r, b] \), parameterized by \( t \in [\delta, b] \).
\bigskip

$\bullet$ A family of infinite-dimensional mean-field stochastic Fredholm-type integral equations  on \( [\delta, b] \), parameterized by \( \tau \in [r, b] \).
\bigskip

We impose the subsequent hypotheses on the generator \( h \):
\bigskip

$(A_{2})$ Let \( r, \delta \in [0, b) \) and \( \eta : [\delta, b] \times [r, b] \times L^2_0 \times L^2_0 \to H \) be a measurable function, where \( s \mapsto \eta(t, s, \mu, \mathbb{E}[\mu]) \) is \( \mathcal{F} \)-progressively measurable for each \( (t, \mu) \in [\delta, b] \times L^2_0 \). 
\bigskip

Furthermore, we assume:

\[
\int_{\delta}^{b} \mathbb{E} \left[ \int_{\tau}^{b} |\eta(t, s, 0, 0)|_H \, ds \right]^2 dt < \infty.
\]
\bigskip

 We impose the subsequent additional conditions on the generator \( h \):

For any \( (t, s) \in [\delta, b] \times [r, b] \) and \( z, \bar{z} \in L^2_0 \), the function \( h \) satisfies the mean field Lipschitz condition almost surely:

\[
|\eta(t, s, z, \mathbb{E}[z]) - \eta(t, s, \bar{z}, \mathbb{E}[\bar{z}])|_H \leq L(t, s) \left( |z - \bar{z}|_{L^2_0} + |\mathbb{E}[z] - \mathbb{E}[\bar{z}]|_{L^2_0} \right),
\]

where \( L : [\delta, b] \times [r, b] \to [0, \infty) \) is a deterministic function such that

\[
\sup_{t \in [\delta, b]} \int_{\tau}^{b} L(t, s)^2 \, ds < \infty.
\]

This condition ensures that \( \eta \) is Lipschitz continuous in both \( z \) and \( \mathbb{E}[z] \), with the continuity bounds controlled by \( L(t, s) \), ensuring integrability over \( [r, b] \).

\begin{lemma}\label{l1}

Assume that condition $(A_{2})$ holds. Then, for any \( \Psi(\cdot) \in L^2_{\mathcal{F}_T}(\delta, b; H) \), equation \eqref{f3.4} has a unique adapted solution 

\[
(\lambda(t, \cdot), \mu(t, \cdot)) \in L^2_{\mathcal{F}}(; C([r, b]; H)) \times L^2_{\mathcal{F}}(r, b; L^2_0)
\]

for almost every \( t \in [\delta, b] \). Moreover, if \( \bar{h} \) also satisfies $(A_{2})$, \( \bar{\Psi}(\cdot) \in L^2_{\mathcal{F}_b}(\delta, b; H) \), and \begin{align*} 
(\bar{\lambda}(t, \cdot), \bar{\mu}(t, \cdot)) \in L^2_{\mathcal{F}}(; C([r, b]; H)) \times L^2_{\mathcal{F}}(r, b; L^2_0) \end{align*} 

is the unique adapted solution to \eqref{f3.4} when \( (h, \Psi) \) is replaced by \( (\bar{h}, \bar{\Psi}) \). Then we have the estimate:

\begin{align}\label{3.5}
&\quad \mathbb{E} \bigg\{ \sup_{\tau \in [r, b]} |\lambda(t, \tau) - \bar{\lambda}(t, \tau)|_H^2 + \int_{\tau}^{b} |\mu(t, s) - \bar{\mu}(t, s)|_{L^2_0}^2 \, ds \bigg\} \nonumber\\
&\leq C \, \mathbb{E} \bigg\{ |\Psi(t) - \bar{\Psi}(t)|_H^2 + \left( \int_{\tau}^{b} \vert \eta(t, s, \mu(t, s), \mathbb{E}[\mu(t, s)]) - \bar{\eta}(t, s, \bar{\mu}(t, s), \mathbb{E}[\bar{\mu}(t, s)])\vert_{H} \, ds \right)^2 \bigg\},
\end{align}

for almost every \( t \in [\delta, b] \).
\end{lemma}

\begin{proof}
Given the structural similarity of this lemma's proof to the arguments employed in Proposition 2.1 and Lemma 3.3 of \cite{Yong}, we omit the detailed derivation here for brevity. The reasoning follows analogous mathematical techniques and can be reconstructed by direct reference to the aforementioned results.
Given the structural similarity of this lemma's proof to the arguments employed in Proposition 2.1 and Lemma 3.3 of \cite{Yong}, we omit the detailed derivation here for brevity. The reasoning follows analogous mathematical techniques and can be reconstructed by direct reference to the aforementioned results.
\end{proof}

This result shows that the difference in solutions \( (\lambda, \mu) \) and \( (\bar{\lambda}, \bar{\mu}) \) depends linearly on the differences between \( \Psi \) and \( \bar{\Psi} \) and between \( \eta \) and \( \bar{\eta} \), with the mean field dependence included in \( \eta \) and \( \bar{\eta} \).

\bigskip

Next, we examine two particular cases of the previous result. First, consider fixing $ \tau=\delta \in [r, b)$. Define

\[
\Psi_{\delta}(t) := \lambda(t, \delta), \quad \aleph(t, s) := \mu(t, s), \quad t \in [r, \delta], \; s \in [\delta, b].
\]

Then equation \eqref{f3.4} becomes

\begin{align}\label{3.6}
\Psi_{\delta}(t) = \Psi(t) + \int_{\delta}^{b} \eta(t, s, \aleph(t, s), \mathbb{E}[\aleph(t, s)]) \, ds - \int_{\delta}^{b} \aleph(t, s) \, dB_{s}, \quad t \in [r, \delta].
\end{align}

This equation represents a Hilbert-space-valued stochastic Fredholm-type integral equation. A pair 

\begin{align*}
(\Psi_{\delta}(\cdot), \aleph(\cdot, \cdot)) \in L^{2}_{\mathcal{F}_{\delta}}(r, \delta; H) \times L^{2}(r, \delta; L^{2}_{\mathcal{F}}(\delta, b; L^{2}_{0}))
\end{align*}

that satisfies \eqref{3.6} in the standard Ito sense is referred to as an adapted solution of \eqref{3.6}. Here, note that \( \Psi_{\delta}(t) \) only needs to be \( \mathcal{F}_{\delta} \)-measurable for almost all \( t \in [r, \delta] \), rather than \( \mathcal{F} \)-adapted. Based on Lemma \ref{l1}, we obtain the subsequent result for this mean-field setting.
\begin{corollary}\label{c1}
Let condition $(A_{2})$ satisfy. Then, for every \( \Psi(\cdot) \in L^2_{\mathcal{F}_{b}}(r, \delta; H) \), the stochastic Fredholm-type integral equation \eqref{3.6} admits a unique adapted solution

\[
(\Psi_{\delta}(\cdot), \aleph(\cdot, \cdot)) \in L^2_{\mathcal{F}_{\delta}}(r, \delta; H) \times L^2(r, \delta; L^2_{\mathcal{F}}(\delta, b; L_0^2)).
\]

The second special case of \eqref{f3.4} is as follows: Let \( r = \delta \), and define

\[
X(t) := \lambda(t, t), \quad t \in [\delta, b], \quad \aleph(t, s) := \mu(t, s), \quad (t, s) \in [\delta, b]^2.
\]

Then equation \eqref{f3.4} becomes

\begin{align}\label{3.7}
X(t) = \Psi(t) + \int_t^b \eta(t, s, \aleph(t, s), \mathbb{E}[\aleph(t, s)]) \, ds - \int_t^b \aleph(t, s) \, dB_{s}, \quad t \in [\delta, b].
\end{align}

This is a specific case of a singular MF-BSVIE \eqref{r3}, where the generator \( \eta \) is independent of \( X(s) \) and \( \aleph(s, t) \). We can define \( \aleph(t, s) \) for \( (t, s) \in [\delta, b]^2 \) by using the martingale representation theorem:

\[
X(t) = \mathbb{E}[X(t) | \mathcal{F}_{\delta}] + \int_{\delta}^t \aleph(t, s) \, dB_{s}, \quad t \in [\delta, b].
\]
\end{corollary}

This formulation provides a mean-field approach to the problem, where the solution process \( X(t) \) is adapted to the filtration \( \mathcal{F}_{\delta} \), and the evolution of \( \aleph(t, s) \) depends on the interactions in the system, modeled through the mean field \( \mathbb{E}[\aleph(t, s)] \).
\begin{corollary}\label{c2}
Let condition $(A_{2})$ hold. Then, for any \( \Psi(\cdot) \in L^2_{\mathcal{F}_b}(\delta, b; H) \), singular MF-BSVIE \eqref{3.7} admits a unique adapted M-solution

\[
(X(\cdot), \aleph(\cdot, \cdot)) \in L^2_{\mathcal{F}}(\delta, b; H) \times L^2(\delta, b; L^2_{\mathcal{F}}(\delta, b; L_0^2)).
\]

Moreover, if \( \bar{\eta} \) also satisfies $(A_{2})$, \( \bar{\Psi}(\cdot) \in L^2_{\mathcal{F}_T}(\delta, b; H) \), and \( (X(\cdot), \aleph(\cdot, \cdot)) \) is the unique adapted M-solution of \eqref{3.7} with \( (\eta, \Psi) \) replaced by \( (\bar{\eta}, \bar{\Psi}) \), then

\begin{align}\label{3.8}
&\quad \mathbb{E} \left[ \sup_{t \in [\delta, b]} \vert X(t) - \bar{X}(t)\vert_H^2 + \int_{\delta}^b \vert \aleph(t, s) - \bar{\aleph}(t, s)\vert_{L_0^2}^2 \, ds \right] \nonumber\\
&\leq C \mathbb{E} \left[ \vert \Psi(t) - \bar{\Psi}(t) \vert_H^2 + \int_t^b \vert \eta(t, s, \aleph(t, s),\mathbb{E}\aleph(t, s)) - \bar{\eta}(t, s, \aleph(t, s),\mathbb{E}\aleph(t, s)) \vert_H^2 \, ds \right],
\end{align}
for \( t \in [\delta, b] \), where \( C \) is a constant depending on the problem's parameters.
\end{corollary}
\bigskip

\begin{theorem}
Assume condition \((A_{1})\) holds. Then for any \(\Psi(\cdot) \in L^2_{\mathcal{F}_T}(0, b; H)\), equation \eqref{r3} has a unique adapted \(M\)-solution within \(H^2[0, b]\). Additionally, the subsequent bound is satisfied:

\[
\mathbb{E} \left\{ \int_0^b |X(t)|_H^2 \, dt + \int_0^b \int_0^b |\aleph(t,s)|_{L^2_0}^2 \, ds \, dt \right\} \leq C \mathbb{E} \int_0^b |\Psi(t)|_H^2 \, dt.
\]

Assume $\bar{F}$ also meets condition $(A_{1})$, and let $\bar{\Psi}(\cdot) \in L^2_{\mathcal{F}_{b}}(0, b; H)$. If \((\bar{Y}(\cdot), \bar{\aleph}(\cdot, \cdot)) \in H^2[0, b]\) represents the adapted \(M\)-solution of \eqref{r3} with \(F\) and \(\Psi(\cdot)\) replaced by \(\bar{F}\) and \(\bar{\Psi}(\cdot)\), respectively, then the subsequent stability estimate is valid:

\begin{align*}
&\quad\mathbb{E} \left\{ \int_0^b |X(t) - \bar{X}(t)|_H^2 \, dt + \int_0^b \int_0^b |\aleph(t,s) - \bar{\aleph}(t,s)|_{L^2_0}^2 \, ds \, dt \right\} \\
&\leq C \bigg\{ \mathbb{E} \int_0^b |\Psi(t) - \bar{\Psi}(t)|_H^2 \, dt \\
& + \mathbb{E} \int_0^b \int_t^b |P(t,s, \bar{X}(s), \bar{\aleph}(t,s), \bar{\aleph}(s,t), \mathbb{E}\bar{X}(s), \mathbb{E}\bar{\aleph}(t,s), \mathbb{E}\bar{\aleph}(s,t)) \\
& - \bar{F}(t,s, \bar{X}(s), \bar{\aleph}(t,s), \bar{\aleph}(s,t), \mathbb{E}\bar{X}(s), \mathbb{E}\bar{\aleph}(t,s), \mathbb{E}\bar{\aleph}(s,t))|_H^2 \, ds \, dt \bigg\}.
\end{align*}
\end{theorem}
 \begin{proof}
 Initially, we introduce \(M^2[0, b]\) as the collection of all pairs \((x(\cdot), z(\cdot, \cdot))\) belonging to \(H^2[0, b]\) such that
     \begin{align*}
x(t)=\mathbb{E}x(s)+\int_{0}^{t}z(t,s)dB_{s},\quad a.e.\quad t\in [0,b].
 \end{align*}
 We define an equivalent norm for \(M^2[0, b]\) in the subsequent manner:
 \begin{align*}
     \Vert (x(\cdot),z(\cdot,\cdot))\Vert_{M^{2}[0,b]} \overset{\triangle}{=}\bigg[\mathbb{E}\int_{0}^{b}\vert x(t)\vert^{2}_{H}dt+\mathbb{E}\int_{0}^{b}\bigg(\int_{t}^{b}\vert z(t,s)\vert^{2}_{L^{0}_{2}}ds\bigg)dt\bigg]^{\frac{1}{2}}.
 \end{align*}

\textbf{ Step 1:}  Let \( (x(\cdot), z(\cdot, \cdot)) \in M^2[\delta, b] \) for some fixed \( \delta \) that is currently undetermined. We will now examine the subsequent equation:
 \begin{align}\label{r4}
X(t)&=\Psi(t)+\int_{t}^{b}P(t,s,x(s),\aleph(t,s),z(s,t),\mathbb{E}\big[x(s)\big], \mathbb{E}\big[\aleph(t,s)\big], \mathbb{E}\big[z(s,t)\big])ds-\int_{t}^{b}\aleph(t,s)dB_{s},
	\end{align}
for any \(\Psi(\cdot) \in L^2_{F_b}(\delta, b; H)\). According to Corollary \ref{c2}, we can observe that the equation \eqref{r4} has a unique adapted \(M\)-solution, denoted as \((X(\cdot), \aleph(\cdot, \cdot))\), and
 \begin{align*}
     &\quad  \mathbb{E}\Bigg\{\vert X(t)\vert^{2}_{H}+\int_{t}^{b}\vert \aleph(t,s)\vert^{2}_{L^{0}_{2}}ds\Bigg\}\\
     &\leq C\mathbb{E}\Bigg\{\vert \Psi(t)\vert^{2}_{H}+\bigg(\int_{t}^{b}\big\vert P(t,s,x(s),0,z(s,t),\mathbb{E}\big[x(s)\big],0,\mathbb{E}\big[z(s,t)\big])\big\vert_{H}ds\bigg)^{2}\Bigg\}\\
      &\leq C\mathbb{E}\Bigg\{\vert \Psi(t)\vert^{2}_{H}+\bigg(\int_{t}^{b}\Big[ L_{x_{1}}(t,s)\vert x(s)\vert_{H}+L_{\xi_{1}}(t,s)\vert z(s,t)\vert_{L^{0}_{2}}\\
      &+ L_{x_{2}}(t,s)\vert \mathbb{E}x(s)\vert_{H}+L_{\xi_{2}}(t,s)\vert \mathbb{E}z(s,t)\vert_{L^{0}_{2}}\Big]ds\bigg)^{2}\Bigg\}\\
       &\leq C\mathbb{E}\Bigg\{\vert \Psi(t)\vert^{2}_{H}+\int_{t}^{b} L_{x_{1}}(t,s)^{2}ds\int_{t}^{b}\vert x(s)\vert^{2}_{H}ds+\int_{t}^{b} L_{\xi_{1}}(t,s)^{2}ds\int_{t}^{b}\vert z(s,t)\vert^{2}_{L^{0}_{2}}ds\\
       &+\int_{t}^{b} L_{x_{2}}(t,s)^{2}ds\int_{t}^{b}\vert \mathbb{E}x(s)\vert^{2}_{H}ds+\int_{t}^{b} L_{\xi_{2}}(t,s)^{2}ds\int_{t}^{b}\vert \mathbb{E}z(s,t)\vert^{2}_{L^{0}_{2}}ds\Bigg\}.
 \end{align*}
 By integrating on $[\delta,b]$ w.r.t. t and using the Jensen's Inequality,  we derive that

 \begin{align}\label{38}
    \Vert (X(\cdot),\aleph(\cdot,\cdot))\Vert^{2}_{M^{2}[\delta,b]} &  \overset{\triangle}{=}\mathbb{E}\bigg\{\int_{\delta}^{b}\vert X(t)\vert^{2}_{H}dt+\int_{\delta}^{b}\bigg(\int_{t}^{b}\vert \aleph(t,s)\vert^{2}_{L^{0}_{2}}ds\bigg)dt\bigg\}\nonumber\\
    &\leq C  \mathbb{E} \bigg\{\int_{\delta}^{b}\vert \Psi(t)\vert^{2}_{H}dt+\int_{\delta}^{b}\int_{t}^{b}L_{x_{1}}(t,s)^{2}dsdt\int_{\delta}^{b}\vert x(t)\vert^{2}_{H}dt\nonumber\\
    &+\sup_{t\in [\delta,b]}\int_{t}^{b}L_{\xi_{1}}(t,s)^{2}ds+\int_{\delta}^{b}\int_{t}^{b}\vert z(s,t)\vert^{2}_{L^{0}_{2}} \, ds dt\nonumber\\
    &+\int_{\delta}^{b}\int_{t}^{b}L_{x_{2}}(t,s)^{2}dsdt\int_{\delta}^{b}\vert \mathbb{E}x(t)\vert^{2}_{H}dt\nonumber\\
    &+\sup_{t\in [\delta,b]}\int_{t}^{b}L_{\xi_{2}}(t,s)^{2}ds+\int_{\delta}^{b}\int_{t}^{b}\vert \mathbb{E}z(s,t)\vert^{2}_{L^{0}_{2}} \, ds dt\bigg\}\\
    &\leq 2C\Big(  \Vert L_{x_{1}}(\cdot,\cdot)\Vert^{2}_{L^{2}(\Delta^{*}[\delta,b])}+ \Vert L_{\xi_{1}}(\cdot,\cdot)\Vert^{2}_{L^{2}(\Delta^{*}[\delta,b])}\nonumber\\
    &+ \Vert L_{x_{2}}(\cdot,\cdot)\Vert^{2}_{L^{2}(\Delta^{*}[\delta,b])}+ \Vert L_{\xi_{2}}(\cdot,\cdot)\Vert^{2}_{L^{2}(\Delta^{*}[\delta,b])}\Big)\nonumber\\
    &\times\bigg[\mathbb{E}\int_{\delta}^{b}\vert \Psi(t)\vert^{2}_{H}dt+\Vert (x(\cdot),z(\cdot,\cdot))\Vert^{2}_{M^{2}[\delta,b]}\bigg].\nonumber
 \end{align}
We define the mapping \(\Theta: M^2[\delta, b] \rightarrow M^2[\delta, b]\) as follows:

\[
\Theta(x(\cdot), z(\cdot, \cdot)) = (X(\cdot), \aleph(\cdot, \cdot)), \quad \text{for all } (x(\cdot), z(\cdot, \cdot)) \in M^2[\delta, b].
\]

Next, we demonstrate that the map \(\Theta\) is contractive for some interval \(0 \leq \delta \leq b\). Consider another pair \((\bar{x}(\cdot), \bar{z}(\cdot, \cdot)) \in M^2[\delta, b]\) such that \(\Theta(\bar{x}(\cdot), \bar{z}(\cdot, \cdot)) = (\bar{X}(\cdot), \bar{\aleph}(\cdot, \cdot))\). By applying the stability estimate in Corollary \ref{c2},

 \[\begin{aligned}
     &\quad \mathbb{E}\int_{\delta}^{b}\vert X(t)-\bar{X}(t)\vert^{2}_{H}dt+\mathbb{E}\int_{\delta}^{b}\int_{t}^{b}\vert \aleph(t,s)-\bar{\aleph}(t,s)\vert^{2}_{L^{0}_{2}}dsdt\\
     &\leq C\mathbb{E}\int_{\delta}^{b}\bigg(\int_{t}^{b}\big\vert P(t,s,x(s),\aleph(t,s),z(s,t),\mathbb{E}\big[x(s)\big],\mathbb{E}\aleph(t,s),\mathbb{E}\big[z(s,t)\big])\\
     &-P(t,s,\bar{x}(s),\aleph(t,s),\bar{z}(s,t),\mathbb{E}\big[\bar{x}(s)\big],\mathbb{E}\aleph(t,s),\mathbb{E}\big[\bar{z}(s,t)\big])\big\vert_{H}ds\bigg)^{2}dt\\
      &\leq C\mathbb{E}\int_{\delta}^{b}\Big[  \int_{t}^{b} L_{x_{1}}(t,s)\vert x(s)-\bar{x}(s)\vert_{H}ds\Big]^{2}dt\\
      &+ C\mathbb{E}\int_{\delta}^{b}\Big[  \int_{t}^{b} L_{\xi_{1}}(t,s)\vert z(s,t)-\bar{z}(s,t)\vert_{L^{0}_{2}}ds\Big]^{2}dt\\
      &+C\mathbb{E}\int_{\delta}^{b}\Big[  \int_{t}^{b} L_{x_{2}}(t,s)\vert \mathbb{E}x(s)-\mathbb{E}\bar{x}(s)\vert_{H}ds\Big]^{2}dt\\
      &+ C\mathbb{E}\int_{\delta}^{b}\Big[  \int_{t}^{b} L_{\xi_{2}}(t,s)\vert \mathbb{E}z(s,t)-\mathbb{E}\bar{z}(s,t)\vert_{L^{0}_{2}}ds\Big]^{2}dt\\
       &\leq C\mathbb{E}\int_{\delta}^{b}\Big[  \int_{t}^{b} L_{x_{1}}(t,s)^{2}ds \int_{t}^{b}\vert x(s)-\bar{x}(s)\vert^{2}_{H}ds\Big]dt\\
        &+ C\mathbb{E}\int_{\delta}^{b}\Big[  \int_{t}^{b} L_{\xi_{1}}(t,s)^{2}ds\int_{t}^{b}\vert z(s,t)-\bar{z}(s,t)\vert^{2}_{L^{0}_{2}}ds\Big]dt\\
      &+C\mathbb{E}\int_{\delta}^{b}\Big[  \int_{t}^{b} L_{x_{2}}(t,s)^{2}ds \int_{t}^{b}\vert \mathbb{E}x(s)-\mathbb{E}\bar{x}(s)\vert^{2}_{H}ds\Big]dt\\
      \end{aligned}\]
      \[\begin{aligned}
      &+ C\mathbb{E}\int_{\delta}^{b}\Big[  \int_{t}^{b} L_{\xi_{2}}(t,s)^{2}ds\int_{t}^{b}\vert \mathbb{E}z(s,t)-\mathbb{E}\bar{z}(s,t)\vert^{2}_{L^{0}_{2}}ds\Big]dt\\
       &\leq C\int_{\delta}^{b}  \int_{t}^{b} L_{x_{1}}(t,s)^{2}dsdt  \mathbb{E} \int_{\delta}^{b}\vert x(s)-\bar{x}(s)\vert^{2}_{H}ds\\
       &+ C\sup_{t\in (\delta,b)} \int_{t}^{b} L_{\xi_{1}}(t,s)^{2}ds 
 \mathbb{E} \int_{\delta}^{b} \int_{\delta}^{t}\vert z(s,t)-\bar{z}(s,t)\vert^{2}_{L^{0}_{2}}dsdt\\
      &+C\int_{\delta}^{b}  \int_{t}^{b} L_{x_{2}}(t,s)^{2}dsdt  \mathbb{E} \int_{\delta}^{b}\vert \mathbb{E}x(s)-\mathbb{E}\bar{x}(s)\vert^{2}_{H}ds\\
      &+ C\sup_{t\in (\delta,b)} \int_{t}^{b} L_{\xi_{2}}(t,s)^{2}ds 
 \mathbb{E} \int_{\delta}^{b} \int_{\delta}^{t}\vert \mathbb{E}z(s,t)-\mathbb{E}\bar{z}(s,t)\vert^{2}_{L^{0}_{2}}dsdt\\
 &\leq C\bigg[\int_{\delta}^{b}  \int_{t}^{b} L_{x_{1}}(t,s)^{2}dsdt+\sup_{t\in (\delta,b)} \int_{t}^{b} L_{\xi_{1}}(t,s)^{2}ds\bigg]\mathbb{E}\int_{\delta}^{b}\vert x(s)-\bar{x}(s)\vert^{2}_{H}ds\\
 &+ C\bigg[\int_{\delta}^{b}  \int_{t}^{b} L_{x_{2}}(t,s)^{2}dsdt+\sup_{t\in (\delta,b)} \int_{t}^{b} L_{\xi_{2}}(t,s)^{2}ds  \bigg]\mathbb{E}\int_{\delta}^{b}\vert \mathbb{E}x(s)-\mathbb{E}\bar{x}(s)\vert^{2}_{H}ds\\
     &\leq C\bigg[\int_{\delta}^{b}  \int_{t}^{b} L_{x_{1}}(t,s)^{2}dsdt+\sup_{t\in (\delta,b)} \int_{t}^{b} L_{\xi_{1}}(t,s)^{2}ds\\
  &+\int_{\delta}^{b}  \int_{t}^{b} L_{x_{2}}(t,s)^{2}dsdt+\sup_{t\in (\delta,b)} \int_{t}^{b} L_{\xi_{2}}(t,s)^{2}ds  \bigg] \mathbb{E}\int_{\delta}^{b}\vert x(s)-\bar{x}(s)\vert^{2}_{H}ds.
 \end{aligned}\]

 Given that \(L_{x_{1}},L_{x_{2}} \in L^2(\Delta^*)\) and \(L_{\xi_1},L_{\xi_2} \in \mathscr{L}^2(\Delta^*)\), we can establish a partition \(\{b_i\}_{i=0}^{m}\) of the interval \([0, b]\), where \(0 = b_0 < b_1 < \ldots < b_m = b\).

\begin{align}\label{k8}
     &C\bigg[\int_{b_i}^{b_{i+1}}  \int_{t}^{b_{i+1}} L_{x_{1}}(t,s)^{2}\, ds\,dt+\sup_{t\in (b_i,b_{i+1})} \int_{t}^{b_{i+1}} L_{\xi_{1}}(t,s)^{2}ds\nonumber\\
     &+\int_{b_i}^{b_{i+1}}  \int_{t}^{b_{i+1}} L_{x_{2}}(t,s)^{2}\, ds\,dt+\sup_{t\in (b_i,b_{i+1})} \int_{t}^{b_{i+1}} L_{\xi_{2}}(t,s)^{2}ds  \bigg]\leq \frac{1}{4},
 \end{align}
 for $i=1,\dots,m-1$.
 In other words, if we set \(\delta = b_{m-1}\), we can rewrite the expression accordingly.

 \begin{align*}
     &C\bigg[\int_{\delta}^{b}  \int_{t}^{b} L_{x_{1}}(t,s)^{2}dsdt+\sup_{t\in (S,b)} \int_{t}^{b} L_{\xi_{1}}(t,s)^{2}ds\\
  &+\int_{\delta}^{b}  \int_{t}^{b} L_{x_{2}}(t,s)^{2}dsdt+\sup_{t\in (S,b)} \int_{t}^{b} L_{\xi_{2}}(t,s)^{2}ds  \bigg]\leq \frac{1}{4}
 \end{align*}

 As a result, the map \(\Theta\) admits a unique fixed point \((X(\cdot), \aleph(\cdot, \cdot)) \in M^2[\delta, b]\), which represents the unique solution to equation \eqref{r3} over the interval \([b_{m-1}, b]\). This step establishes the values of \((X(t), \aleph(t, s))\) for \((t, s) \in [b_{m-1}, b] \times [b_{m-1}, b]\). Furthermore, using results from \eqref{38} and \eqref{k8}, we can derive the subsequent estimate.
 
 \begin{align}\label{11}
     E \left\{ \int_{b_{m-1}}^{b} |X(t)|^2 H \, dt + \int_{b_{m-1}}^{b} \int_{t}^{b} |\aleph(t,s)|^2 L_0^2 \, ds \, dt \right\} \leq C E \int_{b_{m-1}}^{b} |\Psi(t)|^2_{H} \, dt.
 \end{align}
\bigskip

\textbf{Step 2:}
We impose the values \( \aleph(t, s) \) of \( \aleph(\cdot, \cdot) \) for \( (t, s) \in [b_{m-1}, b] \times [b_{m-2}, b_{m-1}] \) using the martingale representation theorem.

\[
E[X(t) | \mathcal{F}_{b_{m-1}}] = E[X(t) | \mathcal{F}_{b_{m-2}}] + \int_{b_{m-2}}^{b_{m-1}} \aleph(t, s) \, dB_{s}, \quad t \in [b_{m-1}, b].
\]
So

\[
E \int_{b_{m-2}}^{b_{m-1}} |\aleph(t, s)|^2_{L_0^2} \, ds \leq E|X(t)|^2_H, \quad t \in [b_{m-1}, b].
\]

Integrating over \( t \in [b_{m-1}, b] \) and using equation \eqref{11}, we obtain:

\begin{align}\label{12}
E \int_{b_{m-1}}^{b} \int_{b_{m-2}}^{b_{m-1}} |\aleph(t, s)|^2_{L_0^2} \, ds \, dt \leq E \int_{b_{m-1}}^{b} |X(t)|^2_H \, dt \leq C E \int_{b_{m-1}}^{b} |\Psi(t)|^2_H \, dt. 
\end{align}

Now, we have imposed \( (X(t), \aleph(t, s)) \) for \( (t, s) \in [b_{m-1}, b] \times [b_{m-2}, b] \).
Combining \eqref{11} and \eqref{12}, we get:

\begin{align}\label{13}
E \left\{ \int_{b_{m-1}}^{b} |X(t)|^2_H \, dt + \int_{b_{m-1}}^{b} \int_{b_{m-2}}^{b_{m-1}} |\aleph(t, s)|^2_{L_0^2} \, ds \, dt \right\} \leq C E \int_{b_{m-1}}^{b} |\Psi(t)|^2_H \, dt.
\end{align}

\bigskip

\textbf{Step 3:}   
For \( (t, s) \in [b_{m-2}, b_{m-1}] \times [b_{m-1}, b] \), we already established in Step 2 that the values of \( X(s) \) and \( \aleph(s, t) \) are known. Here, we concentrate on determining \( \aleph(t, s) \) for \( (t, s) \in [b_{m-2}, b_{m-1}] \times [b_{m-1}, b] \) by solving the corresponding stochastic Fredholm integral equation:

\begin{align}\label{w3}
\Psi^{b_{m-1}}(t) = \Psi(t) + \int_{b_{m-1}}^{b} F^{b_{m-1}}(t, s, \aleph(t, s),\mathbb{E}\aleph(t, s)) \, ds- \int_{b_{m-1}}^{b} \aleph(t, s) \, dB_{s}, \quad t \in [b_{m-2}, b_{m-1}], 
\end{align}

where

\begin{align*}
F^{b_{m-1}}(t, s, z,\mathbb{E}z) \equiv P(t, s, X(s), z, \aleph(s, t),\mathbb{E}X(s), \mathbb{E}z,\mathbb{E}\aleph(s,t)),\\
(t, s, z,\mathbb{E}z) \in [b_{m-2}, b_{m-1}] \times [b_{m-1}, b] \times L_0^2\times L_0^2.
\end{align*}

Thanks to Corollary \ref{c1}, equation \eqref{w3} has a unique adapted solution:

\[
(\Psi^{b_{m-1}}(\cdot), \aleph(\cdot, \cdot)) \in L^2_{\mathcal{F}_{b_{m-1}}}(b_{m-2}, b_{m-1}; H) \times L^2(b_{m-2}, b_{m-1}; L^2_{\mathcal{F}}(b_{m-1}, b; L_0^2)).
\]

Similar to equation \eqref{38}, the subsequent holds:

\begin{align*}
    &\quad\mathbb{E}\bigg\{  \vert \Psi^{b_{m-1}} \vert^{2}_{H}+ \int_{b_{m-1}}^{b} \vert \aleph(t,s)\vert^{2}_{L^{0}_{2}} ds\bigg\}\\
    &\leq C\mathbb{E}\Bigg\{  \vert \Psi \vert^{2}_{H}+ \bigg(\int_{b_{m-1}}^{b} \vert P(t,s,X(s),0,\aleph(s,t),\mathbb{E}X(s),0,\mathbb{E}\aleph(s,t))\vert_{H} ds\bigg)^{2}\Bigg\}\\
    &\leq CE \bigg\{ |\Psi(t)|^2 + \int_{b_{m-1}}^{b} L_{x_{1}}(t, s)^2 \, ds \int_{b_{m-1}}^{b} |X(s)|^2_H \, ds \\
    &+ \int_{b_{m-1}}^{b} L_{\xi_{1}}(t, s)^2 \, ds  \int_{b_{m-1}}^{b} |\aleph(s, t)|^2_{L^0} \, ds\\
    &+ \int_{b_{m-1}}^{b} L_{x_{2}}(t, s)^2 \, ds \int_{b_{m-1}}^{b} |\mathbb{E}X(s)|^2_H \, ds \\
    &+ \int_{b_{m-1}}^{b} L_{\xi_{2}}(t, s)^2 \, ds  \int_{b_{m-1}}^{b} |\mathbb{E}\aleph(s, t)|^2_{L^0} \, ds \bigg\}.
\end{align*}
 Therefore,

\begin{align}\label{98}
   & \quad \mathbb{E}\int_{b_{m-2}}^{b_{m-1}}\bigg\{  \vert \Psi^{b_{m-1}} \vert^{2}_{H}+ \int_{b_{m-1}}^{b} \vert \aleph(t,s)\vert^{2}_{L^{0}_{2}} ds\bigg\}dt\nonumber\\
   &\leq C\mathbb{E}\Bigg\{\int_{b_{m-2}}^{b_{m-1}}  \vert \Psi(t) \vert^{2}_{H}dt+ \int_{b_{m-1}}^{b}  \vert X(t)\vert^{2}_{H}dt+\int_{b_{m-2}}^{b_{m-1}} \int_{b_{m-1}}^{b}  \vert \aleph(s,t)\vert^{2}_{L^{0}_{2}}\,ds\, dt\\
   &+ \int_{b_{m-1}}^{b}  \vert \mathbb{E}X(t)\vert^{2}_{H}dt+\int_{b_{m-2}}^{b_{m-1}} \int_{b_{m-1}}^{b}  \vert \mathbb{E}\aleph(s,t)\vert^{2}_{L^{0}_{2}}dsdt\Bigg\}\leq C\mathbb{E}\int_{b_{m-2}}^{b_{m-1}}  \vert \Psi(t) \vert^{2}_{H}dt.\nonumber
\end{align}
 The final inequality satisfied by the result given in equation \eqref{13}.

 \bigskip

\textbf{Step 4}: We have successfully determined \(X(t)\) for \(t \in [b_{m-1}, b]\), as well as \(\aleph(t,s)\) for the regions \((t,s) \in ([b_{m-1}, b] \times [b_{m-2}, b]) \cup ([b_{m-2}, b_{m-1}] \times [b_{m-1}, b])\). We now focus on the equation 

\begin{align}\label{17}
X(t) &= \Psi^{b_{m-1}}(t) + \int_{t}^{b_{m-1}} P(t,s,X(s),\aleph(t,s),\aleph(s,t),\mathbb{E}X(s),\mathbb{E}\aleph(t,s),\mathbb{E}\aleph(s,t)) \, ds\nonumber\\
&- \int_{t}^{b_{m-1}} \aleph(t,s) dB_{s}, \quad t \in [b_{m-2}, b_{m-1}].
\end{align}

Since \(\Psi^{b_{m-1}}(\cdot)\) is measurable with respect to \(\mathcal{F}_{b_{m-1}}\), \eqref{17} represents a MF-BSVIE over the interval \([b_{m-2}, b_{m-1}]\). By applying the inequality from \eqref{k8} (by selecting \(b_i = b_{m-2}\) and \(b_{i+1} = b_{m-1}\)), and subsequent corresponding techniques as in Step 1, we can establish that \eqref{17} is solvable on the interval \([b_{m-2}, b_{m-1}]\). 
Moreover, according to \eqref{98}, we have 

\begin{align}\label{18}
&\quad \mathbb{E} \left\{ \int_{b_{m-2}}^{b_{m-1}} |X(t)|^2 \, H \, dt + \int_{b_{m-2}}^{b_{m-1}} \int_{b_{m-2}}^{b_{m-1}} |\aleph(t,s)|^2_{L^0_2} \, ds \, dt \right\}\nonumber \\
&\leq C\mathbb{E} \int_{b_{m-2}}^{b_{m-1}} |\Psi^{b_{m-1}}(t)|^2_{H}dt \leq C \mathbb{E} \int_{b_{m-2}}^{b} |\Psi(t)|^2_{H}dt.
\end{align}

This establishes the solvability of \((X(t), \aleph(t,s))\) for the region \((t,s) \in [b_{m-2}, b_{m-1}] \times [b_{m-2}, b_{m-1}]\). As a result, we achieve unique solvability for the MF-BSVIE \eqref{r3} on the interval \([b_{m-2}, b]\). The subsequent estimate can be derived by combining \eqref{13}, \eqref{98}, and \eqref{18}:

\[
\mathbb{E}  \left\{ \int_{b_{m-2}}^{b} |X(t)|^2_{H} dt + \int_{b_{m-2}}^{b} \int_{b_{m-2}}^{b} |\aleph(t,s)|^2_{L^0_2}ds \, dt \right\} \leq C \mathbb{E} \int_{b_{m-2}}^{b} |\Psi(t)|^2_{H}\, dt.
\]

Using corresponding techniques, we can finalize the proof through induction. We now present the stability estimate. Let \((X(\cdot), \aleph(\cdot, \cdot))\) and \((\overline{X}(\cdot), \overline{\aleph}(\cdot, \cdot))\) be adapted \(M\)-solutions of \eqref{r3} corresponding to \((F, \Psi)\) and \((\overline{F}, \overline{\Psi})\), respectively. We define 

\[
\hat{X}(t) \equiv X(t) - \overline{X}(t), \quad \hat{\aleph}(t,s) \equiv \aleph(t,s) - \overline{\aleph}(t,s).
\]

It can be observed that \(\hat{X}(\cdot)\) satisfies the subsequent MF-BSVIE:

\begin{align*}
\hat{X}(t) &= \hat{\Psi}(t) + \int_{t}^{b} \hat{F}(t,s,\hat{X}(s), \hat{\aleph}(t,s), \hat{\aleph}(s,t),\mathbb{E}\hat{X}(s), \mathbb{E}\hat{\aleph}(t,s), \mathbb{E}\hat{\aleph}(s,t)) \, ds - \int_{t}^{b} \hat{\aleph}(t,s) \, dB_{s},
\end{align*}

where 

\begin{align*}
\hat{\Psi}(t)\equiv \Psi(t)-\overline{\Psi}(t) &+\int_{t}^{b} \big[P(t,s,\overline{X}(s), \overline{\aleph}(t,s), \overline{\aleph}(s,t),\mathbb{E}\overline{X}(s), \mathbb{E}\overline{\aleph}(t,s), \mathbb{E}\overline{\aleph}(s,t))\\
&-\overline{F}(t,s,\overline{X}(s), \overline{\aleph}(t,s), \overline{\aleph}(s,t),\mathbb{E}\overline{X}(s), \mathbb{E}\overline{\aleph}(t,s), \mathbb{E}\overline{\aleph}(s,t))\big]ds,
\end{align*}

and 

\begin{align*}
&\quad \hat{F}(t,s,x_{1},z_1,\xi_{1}, x_2, z_2,\xi_{2}) \\
&\equiv P(t,s,x_{1}+\overline{X}(s),z_1+\overline{\aleph}(t,s),\xi_{1}+\overline{\aleph}(s,t),x_{2}+\mathbb{E}\overline{X}(s),z_2+\mathbb{E}\overline{\aleph}(t,s),\xi_{2}+\mathbb{E}\overline{\aleph}(s,t))\\
&- P(t,s,\overline{X}(s), \overline{\aleph}(t,s), \overline{\aleph}(s,t),\mathbb{E}\overline{X}(s), \mathbb{E}\overline{\aleph}(t,s), \mathbb{E}\overline{\aleph}(s,t)).
\end{align*}

It is straightforward to verify that the generator \(\hat{F}\) satisfies the condition in assumption $(A_{1})$. The subsequent stability estimate can be established:

\begin{align*}
&\quad \mathbb{E} \left\{ \int_{0}^{b} \vert \hat{X}(t)\vert^{2}_{H} dt + \int_{0}^{b} \int_{0}^{b} |\hat{\aleph}(t,s)|^2_{L_{0}^{2}}ds dt \right\} \leq C \mathbb{E} \left\{ \int_{0}^{b} |\hat{\Psi}(t)|^{2}_{H} dt \right\} \\
&\leq C \mathbb{E} \bigg\{ \int_{0}^{b} |\Psi(t)-\overline{\Psi}(t)|^2_{H}dt \\
&+ \int_{0}^{b} \bigg(\int_{t}^{b} \vert P(t,s,\overline{X}(s), \overline{\aleph}(t,s), \overline{\aleph}(s,t),\mathbb{E}\overline{X}(s),\mathbb{E} \overline{\aleph}(t,s), \mathbb{E}\overline{\aleph}(s,t)) \\
&-\overline{F}(t,s,\overline{X}(s), \overline{\aleph}(t,s), \overline{\aleph}(s,t),\mathbb{E}\overline{X}(s),\mathbb{E} \overline{\aleph}(t,s), \mathbb{E}\overline{\aleph}(s,t))\vert_{H} ds \bigg)^{2} dt \bigg\}.
\end{align*}
 \end{proof}
\section{Well-posedness of singular MF-FSVIEs}

In the previous section, we established the well-posedness of the singular MF-BSVIE \eqref{r3} using the concept of an adapted M-solution. We found that the singular assumption aligns naturally with the structure of the forward system, which is given by:

\begin{align}\label{4.1}
Y(t) &= \varphi(t) + \int_0^t \Phi(t, s, Y(s), \mathbb{E}[Y(s)]) \, ds+ \int_0^t \Psi(t, s, Y(s), \mathbb{E}[Y(s)]) \, dB_{s}, \quad t \in [0, b],
\end{align}

In this section, we address the solvability of \eqref{4.1}, subject to the conditions outlined below.
\bigskip

$(A_{3}):$
\bigskip

\textbf{$(a)$:} The mappings \( \Phi : \Delta \times H \times \Omega \times H \times \Omega \rightarrow H \) and \( \Psi :\Delta \times H \times \Omega \times H \times \Omega\rightarrow L^0_2 \) are measurable. For each \((t, x, \bar{x}) \in [0, b] \times H \times H\), the mapping \( s \to (\Phi(t, s, x, \bar{x}), \Psi(t, s, x, \bar{x})) \) is \(\mathcal{F}\)-adapted on \([0, t]\).
\bigskip

\textbf{$(b)$:} \( \Phi(t, s, 0, 0) = 0 \) and \( \Psi(t, s, 0, 0) = 0 \) for almost every \((t, s) \in \Delta\), a.s.
\bigskip

\textbf{$(c)$:} There exist functions \( K_1, K_2 \in L^2(\Delta) \) such that for all \( (t, s) \in \Delta \) and \( x, y, \bar{x}, \bar{y} \in H \), we have:
   \begin{align}\label{63}
   \begin{cases}
   |\Phi(t, s, x, \bar{x}) - \Phi(t, s, y, \bar{y})|_H \leq K_1(t, s) \big(|x - y|_H + |\bar{x} - \bar{y}|_H\big),\\
   |\Psi(t, s, x, \bar{x}) - \Psi(t, s, y, \bar{y})|_{L^0_2} \leq K_2(t, s) \big(|x - y|_H + |\bar{x} - \bar{y}|_H\big).
    \end{cases}
   \end{align}

For nonzero mean-field terms \( \Phi(t, s, 0, 0 )\) and \( \Psi(t, s, 0, 0) \), the system  be treated as follows. Actually, assume that $\Phi$ and $\Psi$ hold the Lipschitz condition in Assumption $(A_{3})$, and

\[
\mathbb{E}\int_{0}^{b}\left\{\int_{0}^{t} \vert \Phi(t,s,0,0)\vert_{H}ds\right\}^{2}dt<+\infty,\quad \mathbb{E}\int_{0}^{b}\int_{0}^{t} \vert \Psi(t,s,0,0)\vert^{2}_{L^{0}_{2}}dsdt<+\infty
\]

Then we rewrite \eqref{63} with the mean-field terms as:

\begin{align}\label{64}
Y(t)& = \varphi(t)+\int_0^t \Phi(t, s, 0,0) \, ds  + \int_0^t \tilde{\Phi}(t, s, Y(s), \mathbb{E}[Y(s)]) \, ds \nonumber\\
&+\int_0^t \Psi(t, s, 0,0) \, dB_{s}+ \int_0^t \tilde{\Psi}(t, s, Y(s), \mathbb{E}[Y(s)]) \, dB_{s}, \quad t \in [0, b],
\end{align}
where

\begin{align*}
    \tilde{\Phi}(t, s, y, \bar{y})=\Phi(t,s,y,\bar{y})-\Phi(t,s,0,0),\quad  \tilde{\Psi}(t, s, y, \bar{y})=\Psi(t,s,y,\bar{y})-\Psi(t,s,0,0).
\end{align*}
Such that,  \(\tilde{\Phi}(t, s, 0, 0) = 0\) and \(\tilde{\Psi}(t, s, 0, 0) = 0\), and $\tilde{\Phi}$ and $\tilde{\Psi}$ satisfy the assumption $(A_{3})$. Let the new free term be defined as:  

\[
\tilde{\varphi}(t) = \varphi(t) + \int_0^t \Phi(t, s, 0,0) \, ds + \int_0^t \Psi(t, s, 0,0) \, dB_{s}, \quad t \in [0, b],
\]

where \(\tilde{\varphi}(t)\) incorporates the original term \(\varphi(t)\) along with contributions from the baseline terms \(\Phi(t, s, 0,0)\) and \(\Psi(t, s, 0,0)\).  
Furthermore, we can demonstrate that \(\tilde{\varphi}(\cdot) \in L^2_{\mathcal{F}}(0, b; H)\) if \(\varphi(\cdot) \in L^2_{\mathcal{F}}(0, b; H)\).  
\begin{theorem}\label{t41}
Let condition $(A_{3})$ satisfy. Then, for every \( \varphi(\cdot) \in L^2_{\mathcal{F}}(0, b; H) \), the equation \eqref{63} admits a unique solution \( Y(\cdot) \in L^2_{\mathcal{F}}(0, b; H) \). Moreover, the subsequent estimate holds:
\begin{align}\label{67}
\| Y(\cdot) \|_{L^2_{\mathcal{F}}(0,T;H)} \leq C \| \varphi(\cdot) \|_{L^2_{\mathcal{F}}(0,T;H)}. 
\end{align}

Let \( \Phi^{\prime}\) and \( \Psi^{\prime} \) satisfy condition $(A_{3})$, and let \( \varphi^{\prime}(\cdot) \in L^2_{\mathcal{F}}(0, b; H) \) and \( Y^{\prime}(\cdot) \in L^2_{\mathcal{F}}(0, b; H) \) be the solution of the MF-SVIE \eqref{63} corresponding to \( (\varphi^{\prime}, \Phi^{\prime}, \Psi^{\prime}) \). Then, the subsequent estimate holds:
\begin{align*}
&\quad \bigg[\mathbb{E} \int_0^b \vert Y(t) - Y^{\prime}(t) \vert_H^2 \bigg]^{1/2}\, dt \leq C\bigg\{\bigg[ \mathbb{E} \int_0^b \vert \varphi(t) - \varphi^{\prime}(t) \vert_H^2 \, \\
&+  \int_0^t \vert \Phi(t, s, Y^{\prime}(s), \mathbb{E}[Y^{\prime}(s)]) - \Phi^{\prime}(t, s, Y^{\prime}(s), \mathbb{E}[Y^{\prime}(s)]) \vert_H^2 \, ds \\
&+   \int_0^t \vert \Psi(t, s, Y^{\prime}(s), \mathbb{E}[Y^{\prime}(s)]) - \Psi^{\prime}(t, s, Y^{\prime}(s), \mathbb{E}[Y^{\prime}(s)]) \vert_H^2 \, ds\bigg] dt\bigg\}^{1/2}.
\end{align*}
\end{theorem}
\begin{proof}
    We apply the Banach fixed-point theorem to establish the existence and uniqueness of the adapted solution for equation \eqref{63}. In the steps that follow, we use the notation \(\bar{x} = \mathbb{E}x\) and \(\bar{y} = \mathbb{E}y\).

 \textbf{Step 1:}  For any \(\delta \in (0, b]\) and \(y(\cdot) ,\bar{y}(\cdot)\in L^2_{\mathcal{F}}(0, \delta; H)\), we define the mapping:

\[
S[y(\cdot)](t) = \varphi(t) + \int_0^t \Phi(t, s, y(s),\bar{y}(s)) \, ds + \int_0^t \Psi(t, s, y(s),\bar{y}(s)) \, dB_{s}, \quad t \in [0, \delta].
\]

From the definition of \(S[y(\cdot)](t)\), we derive the subsequent bound:

\begin{align}
&\quad\mathbb{E} \int_0^{\delta} |S[y(\cdot)](t)|_H^2 \, dt\nonumber\\
&\leq C\Bigg[ \mathbb{E} \int_0^{\delta} |\varphi(t)|_H^2 \, dt + \mathbb{E} \int_0^{\delta} \left( \int_0^t K_1(t, s) \left(|y(s)|_H+|\bar{y}(s)|_H\right) \, ds \right)^2 dt\nonumber \\
&+ \mathbb{E} \int_0^{\delta} \int_0^t K_2(t, s)^2 \left(|y(s)|_H+|\bar{y}(s)|_H\right)^2 \, ds \, dt\Bigg] \nonumber
\end{align}

Using Fubini's theorem and the Cauchy-Schwarz inequality, we further simplify:

\begin{align}\label{66}
& \quad\mathbb{E} \int_0^{\delta} |S[y(\cdot)](t)|_H^2 \, dt\nonumber\\
&\leq C \Bigg[ \mathbb{E} \int_{0}^{\delta} \vert\varphi(t)\vert_H^2 \, dt + \int_{0}^{\delta} \int_{0}^{t} K_1(t, s)^2 \,ds \,dt\, \mathbb{E} \int_{0}^{\delta} \left[ \left(|y(s)|_H + |\bar{y}(s)|_H\right)^2 \right]  ds\nonumber\\
&+ \mathbb{E}\int_{0}^{\delta} \int_{s}^{\delta} K_2(t, s)^2 \left[ \left(|y(s)|_H + |\bar{y}(s)|_H\right)^2 \right]   ds \,dt\Bigg]
\leq C \mathbb{E} \int_{0}^{\delta} |\varphi(t)|_H^2 \, dt\nonumber\\
&+ 2\left(\|K_1\|_{L^2(\Delta)} + \|K_2\|_{L^2(\Delta)}\right) \left(\mathbb{E} \int_{0}^{\delta} |y(t)|_H^2 \, dt+\mathbb{E} \int_{0}^{\delta} |\bar{y}(t)|_H^2 \, dt\right)\nonumber\\
&\leq C \mathbb{E} \int_{0}^{\delta} |\varphi(t)|_H^2 \, dt+ 4\left(\|K_1\|_{L^2(\Delta)} + \|K_2\|_{L^2(\Delta)}\right) \mathbb{E} \int_{0}^{\delta} |y(t)|_H^2 \, dt
\end{align}
 Hence, $S$ maps $L^2_{F}(0, \delta; H)$ to itself.
\bigskip

\textbf{Step 2:} 
Let \( b_{1} \in (0, b] \) be chosen later. Our goal is to demonstrate that the mapping \( S \) defined earlier is contractive on the interval \([0, b_{1}]\). For any \( x(\cdot), \bar{x}(\cdot), y(\cdot), \bar{y}(\cdot) \in L^{2}_{\mathcal{F}}(0, b_{1}; H) \), by applying Hölder's inequality and Fubini's theorem, we derive:

\begin{align*}
&\quad \|S[x(\cdot)] - S[y(\cdot)]\|_{L^2_{\mathcal{F}}(0, b_1; H)}^2 \\
&\leq 2 \mathbb{E} \int_0^{b_1} \left( \int_0^t K_1(t, s) \left(|x(s) - y(s)|_H+|\bar{x}(s) - \bar{y}(s)|_H\right) \, ds \right)^2 dt \\
&+ 2 \mathbb{E} \int_0^{b_1} \int_0^t K_2(t, s)^2 \left(|x(s) - y(s)|_H+|\bar{x}(s) - \bar{y}(s)|_H\right)^2 \, ds \, dt\\
&\leq 2 \int_0^{b_1} \int_0^t K_1(t, s)^2 \, ds \, dt \, \mathbb{E} \int_0^{b_1} \left(|x(s) - y(s)|_H+|\bar{x}(s) - \bar{y}(s)|_H\right)^2 \, ds \\
&+ 2 \mathbb{E} \int_0^{b_1} \int_s^{b_1} K_2(t, s)^2 \, dt  \,\left(|x(s) - y(s)|_H+|\bar{x}(s) - \bar{y}(s)|_H\right)^2 \, ds\\
&\leq 4 \int_0^{b_1}  \int_0^t K_1(t, s)^2 \, ds \, dt \, \mathbb{E} \int_0^{b_1} |x(s) - y(s)|^{2}_H\, ds \\
&+4 \int_0^{b_1}  \int_0^t K_1(t, s)^2 \, ds \, dt \, \mathbb{E} \int_0^{b_1} |\bar{x}(s) - \bar{y}(s)|_{H}^2 \, ds \\
&+ 4 \mathbb{E} \int_0^{b_1} \int_s^{b_1} K_2(t, s)^2 \, dt  |x(s) - y(s)|^{2}_{H} \, ds\\
&+ 4 \mathbb{E} \int_0^{b_1}  \int_s^{b_1} K_2(t, s)^2 \, dt  |\bar{x}(s) - \bar{y}(s)|^{2}_{H} \, ds\\
&\leq 8 \int_0^{b_1}  \int_0^t K_1(t, s)^2 \, ds \, dt \, \mathbb{E} \int_0^{b_1} |x(s) - y(s)|^{2}_H \, ds \\
&+ 8\, \mathbb{E} \int_0^{b_1} \int_s^{b_1} K_2(t, s)^2 \, dt \, |x(s) - y(s)|^{2}_{H} \, ds
\end{align*}

From the definitions of \( K_1 \) and \( K_2 \), we can construct a partition \( \{b_i\}_{i=0}^m \) of \([0, b]\) such that \( 0 = b_0 < b_1 < \cdots < b_m = b \), and for all \( i = 0, 1, \dots, m-1 \):
\begin{align}\label{65}
\int_{b_i}^{b_{i+1}} \int_{b_i}^t K_1(t, s)^2 \, ds \, dt + \sup_{s \in (b_i, b_{i+1})} \int_{s}^{b_{i+1}} K_2(t, s)^2 \, dt \leq \frac{1}{16}.
\end{align}

In particular, for \( i = 0 \), this implies:
\[
\int_0^{b_1} \int_0^t K_1(t, s)^2 \, ds \, dt + \sup_{s \in (0, b_1)} \int_s^{b_1} K_2(t, s)^2 \, dt \leq \frac{1}{16}.
\]

Thus, \( S \) is contractive on \( L^2_{\mathcal{F}}(0, b_1; H) \). Therefore, the equation in \eqref{63} has a unique solution \( Y_1(\cdot) \in L^2_{\mathcal{F}}(0, b_1; H) \) on \([0, b_1]\). Moreover, using the estimate in \eqref{66} and the condition in \eqref{65}, the solution \( Y(\cdot) \) satisfies the bound in \eqref{67}, with \( b\) replaced by \( b_1 \).

\bigskip

\textbf{Step 3:} We now rewrite equation \eqref{63} as:
\begin{align*}
Y(t) &= \hat{\varphi}(t) + \int_{b_1}^t \Phi(t, s, Y(s), \mathbb{E}[Y(s)]) \, ds + \int_{b_1}^t \Psi(t, s, Y(s), \mathbb{E}[Y(s)]) \, dB_{s}, \quad t \in [b_1, b],
\end{align*}
where
\begin{align*}
\hat{\varphi}(t)& = \varphi(t) + \int_0^{b_1} \Phi(t, s, Y_1(s), \mathbb{E}[Y_1(s)]) \, ds
+ \int_0^{b_1} \Psi(t, s, Y_1(s), \mathbb{E}[Y_1(s)]) \, dB_{s}, \quad t \in [b_1, b].
\end{align*}
It can be easily shown that \( \hat{\varphi}(t) \in L^2_{\mathcal{F}}(b_1, b; H) \). For any \( y(\cdot) \in L^2_{\mathcal{F}}(b_1, b_2; H) \), we define
\begin{align*}
S[y(\cdot)](t)& = \hat{\varphi}(t) + \int_{b_1}^t \Phi(t, s, y(s), \mathbb{E}[y(s)]) \, ds + \int_{b_1}^t \Psi(t, s, y(s), \mathbb{E}[y(s)]) \, dB_{s}, \quad t \in [b_1, b_2].
\end{align*}
By applying corresponding techniques as in Step 1 and 2, and considering the condition \eqref{65}, we can obtain the unique solution of equation \eqref{63} on the interval \([b_1, b_2]\). Repeating these steps, we can solve on the entire interval \([0, b]\), and the estimate \eqref{67} holds.

Now, define the difference \( \bar{Y}(\cdot) = Y(\cdot) - Y(\cdot) \), which belongs to \( L^2_{\mathcal{F}}(0, b; H) \), and satisfies the subsequent MF-SVIE:
\begin{align*}
\bar{Y}(t) &= \bar{\varphi}(t) + \int_{0}^{t} \bar{\Phi}(t, s, \bar{Y}(s), \mathbb{E}[\bar{Y}(s)]) \, ds + \int_{0}^{t} \bar{\Psi}(t, s, \bar{Y}(s), \mathbb{E}[\bar{Y}(s)]) \, dB_{s}, \quad t \in [0, b],
\end{align*}
where
\begin{align}\label{68}
\bar{\varphi}(t) = \varphi(t) - \varphi(t) + \int_0^t \left( \Phi(t, s, Y(s), \mathbb{E}[Y(s)]) - \Phi(t, s, Y(s), \mathbb{E}[Y(s)]) \right) \, ds\nonumber\\
+ \int_0^t \left( \Psi(t, s, Y(s), \mathbb{E}[Y(s)]) - \Psi(t, s, Y(s), \mathbb{E}[Y(s)]) \right) \, dB_{s},
\end{align}
and
\[
\bar{\Phi}(t, s, y, \mathbb{E}[y]) = \Phi(t, s, y, \mathbb{E}[y]) + Y(s) - \Phi(t, s, Y(s), \mathbb{E}[Y(s)]),
\]
\[
\bar{\Psi}(t, s, y, \mathbb{E}[y]) = \Psi(t, s, y, \mathbb{E}[y]) + Y(s) - \Psi(t, s, Y(s), \mathbb{E}[Y(s)]).
\]
Since \( \bar{\varphi}(t) \in L^2_{\mathcal{F}}(0, b; H) \) and \( \bar{A} \) and \( \bar{B} \) satisfy the Lipschitz condition $(A_{3})$ with constants \( K_1 \) and \( K_2 \), we apply estimate \eqref{67} to obtain:
\[
\mathbb{E} \int_0^b |\bar{Y}(t)|_H^2 \, dt \leq C \mathbb{E} \int_0^b |\bar{\varphi}(t)|_H^2 \, dt,
\]
which gives the stability estimate, including the mean field terms.
\end{proof}

\section{The Maximum Principle for Controlled Stochastic Volterra Integral Equations with a Mean-Field Term}  

This section explores the application of optimal control theory to controlled stochastic Volterra integral equations incorporating a mean-field term. Our goal is to derive the maximum principles for optimal control under the assumption of convex control domains. Drawing inspiration from fractional stochastic evolution equations and stochastic evolutionary integral equations, we examine an optimal control problem defined by a mean-field stochastic Volterra integral equation, with the corresponding cost functional presented in a Lagrangian form.
\bigskip

The maximum principle is a cornerstone of optimal control theory, offering necessary conditions for optimality in systems governed by stochastic dynamics. It establishes a critical connection between the dynamics of the state equation and the cost functional through an associated adjoint process, providing a framework for identifying optimal strategies. Its significance extends beyond theoretical foundations to practical applications across various domains:
\bigskip

 The maximum principle is pivotal in designing optimal control strategies for systems characterized by memory effects and mean-field interactions. Such systems are common in fields like engineering and robotics, where managing complex and dynamic systems effectively is essential to achieve long-term objectives.
\bigskip

By deriving the maximum principle for controlled mean-field stochastic Volterra integral equations, this study enhances the theoretical understanding and practical application of optimal control in systems influenced by both stochastic factors and memory-dependent behaviors. Future research is expected to generalize these results to more intricate systems and develop computational approaches for solving these equations in real-world scenarios.

\bigskip
We now proceed to the specific state equation and cost functional that define the optimal control problem. The state equation is given by:

\begin{align}\label{h1}
Y(t) &= \varphi(t) + \int_0^t \kappa(t, s, Y(s), \mathbb{E}[Y(s)],u(s)) \, ds + \int_0^t \nu(t, s, Y(s), \mathbb{E}[Y(s)], u(s)) \, dB_{s}, \quad t \in [0,b],
\end{align}

and the cost functional is expressed as:

\begin{align}\label{h2}
J(u(\cdot)) = \mathbb{E} \left[ \int_0^b g(t, Y(t), \mathbb{E}[Y(t)], u(t)) \, dt \right].
\end{align}
Here:
\begin{itemize}
    \item \( u(\cdot) \) represents the control process, taking values within a non-empty convex set \( U \), which is a subset of a separable metric space \( S \).
    \item \( Y(\cdot) \) denotes the state process, which evolves within the space \( H \).
    \item The function \( \varphi : [0, b] \to H \) acts as the free term in the system.
    \item The drift term is described by the function \( \kappa : [0, b] \times H \times H \times U \to H \), incorporating the interaction of time, state, and control.
    \item The diffusion term is modeled by \( \nu : [0, b] \times H \times H \times U \to L_0^2 \), representing stochastic effects in the system.
    \item The running cost \( g : [0, b] \times H \times H \times U \to \mathbb{R} \) represents the cost function, which includes the mean-field term \( \mathbb{E}[Y(t)] \).
\end{itemize}

We now introduce the subsequent assumptions related to the problem.

\bigskip

$(A_{3})$ Let \( \kappa \), \( \nu \), and \( g \) be continuously differentiable functions in the variables \( (y,\mathbb{E}y, u) \), where \( y \) and  $\mathbb{E}y$ are elements of \( H \) (a Hilbert space) and \( u \) is an element of \( U \) (a control space). Additionally, suppose that the Fréchet derivatives \( g_y \) with respect to \( y \), \( g_{\mathbb{E}y} \) with respect to \( \mathbb{E}y \)  and \( g_u \) with respect to \( u \) are bounded.
\bigskip

The subsequent conditions are assumed to hold:
\bigskip

$\bullet$ For every \( (t, s) \in \mathcal{A} \), \( y \in H \), and \( u \in U \), the terms \( |\kappa(t, s, 0,0, u)|_H + |\nu(t, s, 0,0, u)|_{L^2_0}\leq C \),
  \bigskip

$\bullet$ The derivatives of \( \kappa \) with respect to \( y \) and \( u \) satisfy the subsequent bounds:
  \[
  |\kappa_y(t, s, y,\mathbb{E}y, u)|_{L(H)} + |\kappa_u(t, s, y,\mathbb{E}y, u)|_{L(U; H)} \leq K_1(t, s),
  \]
  where \( K_1(t, s) \) is a function that belongs to \( L^2(\mathcal{A}) \),
\bigskip

$\bullet$ The derivatives of \( \nu \) with respect to \( y \) and \( u \) satisfy the subsequent bounds:
  \[
  |\nu_y(t, s, y,\mathbb{E}y, u)|_{L(H; L^2_0)} + |\nu_u(t, s, y,\mathbb{E}y, u)|_{L(U; L^2_0)} \leq K_2(t, s),
  \]
  where \( K_2(t, s) \) is also a square-integrable function over \( \mathcal{A} \). 
\bigskip

Let \( U[0, b] \subset L^2_{F}(0, b; U) \) represent the set of admissible controls. Under the condition $(A_{3})$, by Theorem 3.10, for every \( \varphi(\cdot) \in L^2_{F}(0, b; H) \) and \( u(\cdot) \in U[0, b] \), the equation (4.1) has a unique solution \( Y(\cdot) \in L^2_{F}(0, b; H) \). Therefore, the cost functional \( J(u(\cdot)) \) is well-defined, and the optimal control problem is formulated as outlined below:
\bigskip
\begin{table}[h!]
\renewcommand{\arraystretch}{1.5}
\centering
\begin{tabular}{|p{0.9\textwidth}|}
\hline
\textbf{Problem (V)}: The objective is to find a control function \( \bar{u}(\cdot) \in U[0, b] \) such that: \\[5pt]
\[
J(\bar{u}(\cdot)) = \inf_{u(\cdot) \in U[0,b]} J(u(\cdot)),
\]
where \( J(u(\cdot)) \) denotes the cost functional. The function \( \bar{u}(\cdot) \) that achieves this minimum is called the \textit{optimal control}. The corresponding state process, represented by \( \bar{Y}(\cdot) \), and the pair \( (\bar{Y}(\cdot), \bar{u}(\cdot)) \) are referred to as the \textit{optimal state} and the \textit{optimal pair}, respectively. \\ \hline
\end{tabular}
\end{table}
\begin{theorem}
 
Let \((\bar{X}(\cdot), \bar{u}(\cdot))\) be an optimal pair. Then, the subsequent MF-BSVIE admits a unique adapted \(M\)-solution \((X(\cdot), \aleph(\cdot, \cdot))\):   

\begin{align}\label{1}
  X(t)&=g_{y}(t,\bar{Y}(t),\mathbb{E}\bar{Y}(t),\bar{u}(t))+\mathbb{E}g_{\mathbb{E}y}(t,\bar{Y}(t),\mathbb{E}\bar{Y}(t),\bar{u}(t))\nonumber\\
  &+\int_{t}^{b}\Big(\kappa_{y}(t,s,\bar{Y}(s),\mathbb{E}\bar{Y}(s),\bar{u}(s))^{*}X(s)+\nu_{y}(t,s,\bar{Y}(s),\mathbb{E}\bar{Y}(s),\bar{u}(s))^{*}\aleph(s,t)Big)ds\nonumber\\
   &+\mathbb{E}\int_{t}^{b}\Big(\kappa_{\mathbb{E}y}(t,s,\bar{Y}(s),\mathbb{E}\bar{Y}(s),\bar{u}(s))^{*}X(s) +\nu_{\mathbb{E}y}(t,s,\bar{Y}(s),\mathbb{E}\bar{Y}(s),\bar{u}(s))^{*}\aleph(s,t)\Big)ds\nonumber\\
   &-\int_{t}^{b}\aleph(t,s)dB_{s},\quad t\in [0,b].
\end{align}
such that the subsequent inequality holds almost surely: 
\begin{align}\label{2}
&\Big\langle g_{u}(t, \bar{Y}(t),\mathbb{E}\bar{Y}(t), \bar{u}(t)) + \mathbb{E} \bigg[ \int_t^b \kappa_u(s, t, \bar{Y}(t),\mathbb{E}\bar{Y}(t), \bar{u}(t))^\ast X(s) \, ds\nonumber \\
&+ \int_t^b \nu_u(s, t, \bar{Y}(t),\mathbb{E}\bar{Y}(t), \bar{u}(t))^\ast \aleph(s, t) \, ds \mid \mathscr{F}_t \bigg](u - \bar{u}(t))\Big\rangle_{U} \geq 0,\quad \forall u\in U, \quad  t\in [0,b]\, a.s.
\end{align}
\end{theorem}
\begin{proof}
To begin with, we introduce the subsequent abbreviations:

\begin{align*}
\begin{cases}
     \kappa_{y}(t, s) = \kappa_{y}(t, s,\bar{Y}(s), \mathbb{E}\bar{Y}(s),\bar{u}(s)), \\
      \kappa_{\mathbb{E}y}(t, s) = \kappa_{\mathbb{E}y}(t, s,\bar{Y}(s), \mathbb{E}\bar{Y}(s),\bar{u}(s)), \\
     \kappa_{u}(t, s) = \kappa_u(t, s,\bar{Y}(s), \mathbb{E}\bar{Y}(s),\bar{u}(s)), \\
\end{cases}
\begin{cases}
    \nu_y(t, s) = \nu_y(t, s,\bar{Y}(s), \mathbb{E}\bar{Y}(s),\bar{u}(s)),\\
     \nu_{\mathbb{E}y}(t, s) = \nu_{\mathbb{E}y}(t, s,\bar{Y}(s), \mathbb{E}\bar{Y}(s),\bar{u}(s)),\\ 
     \nu_{u}(t, s) = \nu_u(t, s,\bar{Y}(s), \mathbb{E}\bar{Y}(s),\bar{u}(s)), \\
\end{cases}
\end{align*}
\begin{align*}
    \begin{cases}
    g_{y}(t) = g_y(t, s,\bar{Y}(s), \mathbb{E}\bar{Y}(s),\bar{u}(s)), \\
      g_{\mathbb{E}y}(t) = g_{\mathbb{E}y}(t, s,\bar{Y}(s), \mathbb{E}\bar{Y}(s),\bar{u}(s)),\\
      g_{\bar{u}}(t) = g_u(t, s,\bar{Y}(s), \mathbb{E}\bar{Y}(s),\bar{u}(s)).
    \end{cases}
\end{align*}
Let \( (\bar{Y}(\cdot), \bar{u}(\cdot)) \) represent an optimal solution to Problem (V). For any control function \( v(\cdot) \in \mathscr{U}[0, b] \) and any real number \( \varepsilon \), define the perturbed control \( u_{\varepsilon}(\cdot) \) as:

\[
u^{\varepsilon}(\cdot) = \bar{u}(\cdot) + \varepsilon [v(\cdot) - \bar{u}(\cdot)]\in \mathscr{U}[0,b].
\]
Let \( Y^{\varepsilon}(\cdot) \) be the solution of equation \eqref{h1} when the control \( u(\cdot) \) is replaced by the perturbed control \( u^{\varepsilon}(\cdot) \). Then,
\begin{align*}
&\quad Y^{\varepsilon}(t) - \bar{Y}(t) 
= \int_0^t \bigg[ \tilde{\kappa}_y(t, s) [Y^{\varepsilon}(s) - \bar{Y}(s)] \\
&+ \tilde{\kappa}_{\mathbb{E}y}(t, s) [\mathbb{E}Y^{\varepsilon}(s) - \mathbb{E}\bar{Y}(s)]  + \tilde{\kappa}_{u}(t, s) \varepsilon [v(s) - \bar{u}(s)]\bigg] ds\\
&+\int_0^t \bigg[ \tilde{\nu}_y(t, s) [Y^{\varepsilon}(s) - \bar{Y}(s)]\\
&+ \tilde{\nu}_{\mathbb{E}y}(t, s) [\mathbb{E}Y^{\varepsilon}(s) - \mathbb{E}\bar{Y}(s)]  + \tilde{\nu}_{u}(t, s) \varepsilon [v(s) - \bar{u}(s)]\bigg] dB_{s},
\end{align*}

where

\[
\tilde{\kappa}_{y}(t,s) = \int_0^1 \kappa_y(t,s,\bar{Y}(s) + \gamma[\bar{Y}^\varepsilon(s) - \bar{Y}(s)],\mathbb{E}\bar{Y}(s) ,u^\varepsilon(s)) \, d\gamma,
\]
\[
\tilde{\kappa}_{\mathbb{E}y}(t,s) = \int_0^1 \kappa_y(t,s,\bar{Y}(s),\mathbb{E}\bar{Y}(s) + \beta[\mathbb{E}\bar{Y}^\varepsilon(s) - \mathbb{E}\bar{Y}(s)] ,u^\varepsilon(s)) \, d\gamma,
\]
\[
\tilde{\kappa}_{u}(t,s) = \int_0^1 \kappa_u(t,s,\bar{Y}(s),\mathbb{E}\bar{Y}(s) ,\bar{u}(s) + \theta\varepsilon v(s)]) \, d\theta,
\]
\[
\tilde{\nu}_{y}(t,s) = \int_0^1 \nu_y(t,s,\bar{Y}(s) + \gamma[\bar{Y}^\varepsilon(s) - \bar{Y}(s)],\mathbb{E}\bar{Y}(s) ,u^\varepsilon(s)) \, d\gamma,
\]
\[
\tilde{\nu}_{\mathbb{E}y}(t,s) = \int_0^1 \nu_y(t,s,\bar{Y}(s),\mathbb{E}\bar{Y}(s) + \beta[\mathbb{E}\bar{Y}^\varepsilon(s) - \mathbb{E}\bar{Y}(s)] ,u^\varepsilon(s)) \, d\gamma,
\]
\[
\tilde{\nu}_{u}(t,s) = \int_0^1 \nu_u(t,s,\bar{Y}(s),\mathbb{E}\bar{Y}(s) ,\bar{u}(s) + \theta\varepsilon v(s)]) \, d\theta.
\]

Using \eqref{67} from Theorem \ref{t41}, we obtain:

\begin{align*}
\Vert Y^{\varepsilon}(\cdot)- \bar{Y}(\cdot)\Vert^{2}_{L^{2}_{F}(0,T;H)}&=\mathbb{E}\int_{0}^{b}\bigg\vert \int_{0}^{t} \bigg[ \tilde{\kappa}_y(t, s) [Y^{\varepsilon}(s) - \bar{Y}(s)] \\
&+ \tilde{\kappa}_{\mathbb{E}y}(t, s) [\mathbb{E}Y^{\varepsilon}(s) - \mathbb{E}\bar{Y}(s)]  + \tilde{\kappa}_{u}(t, s) \varepsilon [v(s) - \bar{u}(s)]\bigg] ds\\
&+\int_0^t \bigg[ \tilde{\nu}_y(t, s) [Y^{\varepsilon}(s) - \bar{Y}(s)] + \tilde{\nu}_{\mathbb{E}y}(t, s) [\mathbb{E}Y^{\varepsilon}(s) - \mathbb{E}\bar{Y}(s)]  \\
&+ \tilde{\nu}_{u}(t, s) \varepsilon [v(s) - \bar{u}(s)]\bigg] dB_{s}\bigg\vert^{2}_{H}dt\leq C\varepsilon^{2}.
\end{align*}
Then, we obtain
\begin{align*}
    \lim_{\varepsilon\to 0}\Vert Y^{\varepsilon}(\cdot)- \bar{Y}(\cdot)\Vert^{2}_{L^{2}_{F}(0,T;H)}=o(\varepsilon) .
\end{align*}
Define
\begin{align*}
    Y^{\varepsilon}_{1}(t)=\frac{Y^{\varepsilon}(t)- \bar{Y}(t)}{\varepsilon},\quad t\in [0,b].
\end{align*}
The function \( Y^\varepsilon_{1}(\cdot) \) converges to \( Y_{1}(\cdot) \) in \( L^2_{\mathscr{F}}(0, b; H) \), where \( Y_{1}(\cdot) \) holds the subsequent equation:  

\begin{align}\label{3}
Y_{1}(t) &= \int_0^t \big[\kappa_y(t, s) Y_{1}(s)+\kappa_{\mathbb{E}y}(t, s) \mathbb{E}Y_{1}(s) + \kappa_u(t, s) \big(v(s) - \bar{u}(s)\big)\big] \, ds \nonumber\\
&+ \int_0^t \big[\nu_y(t, s) Y_{1}(s) +\nu_{\mathbb{E}y}(t, s) \mathbb{E}Y_{1}(s)+ \nu_u(t, s) \big(v(s) - \bar{u}(s)\big)\big] \, dB_{s}\\
 &= \bar{\varphi}(t) + \int_0^t \kappa_y(t, s) Y_{1}(s) \, ds + \int_0^t \nu_y(t, s) Y_{1}(s) \, dB_{s}\nonumber\\
&+ \int_0^t \kappa_{\mathbb{E}y}(t, s) \mathbb{E}Y_{1}(s) \, ds + \int_0^t \nu_{\mathbb{E}y}(t, s) \mathbb{E}Y_{1}(s)dB_{s}, \quad t \in [0, b],\nonumber
\end{align}

where \( \bar{\varphi}(t) \) represents the remaining terms of the equation after separating the integral components.

\begin{align*}
    \bar{\varphi}(t)=\int_0^t  \kappa_{u}(t, s) \big(v(s) - \bar{u}(s)\big) \, ds + \int_0^t  \nu_{u}(t, s) \big(v(s) - \bar{u}(s)\big) \, dB_{s},
\end{align*}

By the optimality of \((\bar{Y}(\cdot), \bar{u}(\cdot))\), we have the variational inequality:  

\begin{align*}
   & 0\leq \frac{J(u^{\varepsilon}(\cdot))-J(\bar{u}(\cdot))}{\varepsilon}\\
&=\mathbb{E}\int_{0}^{b}\bigg[\Big\langle  g_{y}(t,\bar{Y}(t)+\gamma (Y^{\varepsilon}(t)-\bar{Y}(t)) , \mathbb{E}\bar{Y}(t), u^{\varepsilon}(t), \frac{Y^{\varepsilon}(t)-\bar{Y}(t)}{\varepsilon}\Big\rangle_{H}\\
    &+\Big\langle  g_{\mathbb{E}y}(t,\bar{Y}(t), \mathbb{E}\bar{Y}(t)+\beta (\mathbb{E}Y^{\varepsilon}(t)-\mathbb{E}\bar{Y}(t)) , u^{\varepsilon}(t), \frac{\mathbb{E} Y^{\varepsilon}(t)-\mathbb{E} \bar{Y}(t)}{\varepsilon}\Big\rangle_{H}\\
    &+\Big\langle  g_{\mathbb{E}y}(t,\bar{Y}(t), \mathbb{E}\bar{Y}(t), \bar{u}+\theta \varepsilon (v(t)-\bar{u}(t)), v(t)-\bar{u}(t))\Big\rangle_{U}\bigg]dt\\
    &\to \mathbb{E}\int_{0}^{b} \big[\langle g_{y}(t),Y_{1}(t)\rangle_{H}+\langle g_{\mathbb{E}y}(t),\mathbb{E}Y_{1}(t)\rangle_{H}+\langle g_{u}(t), v(t)-\bar{u}(t))\rangle_{U}\big]dt\\
    &=J_{1}+J_{2}+J_{3},
\end{align*}

where $\gamma,\beta,\theta\in (0,1)$ and 

\[
I_1 \coloneqq \mathbb{E} \int_0^b \big\langle g_{y}(t), Y_1(t) \big\rangle_{H} \, dt, \quad I_2 \coloneqq \mathbb{E} \int_0^b \big\langle g_{\mathbb{E}y}(t), \mathbb{E}Y_1(t) \big\rangle_{H} \, dt,\quad  
I_3 \coloneqq  \mathbb{E}  \int_0^b \big\langle g_u(t), (v(t) - \bar{u}(t)) \big\rangle_{U} \, dt.
\]

To eliminate the term involving \( Y_1(\cdot) \), we obtain 

\[\begin{aligned}
    &\quad \mathbb{E}\int_{0}^{b}\langle \bar{\varphi}(t), X(t)\rangle_{H} dt\\
    &=\mathbb{E}\int_{0}^{b}\Big[\Big\langle Y_{1}(t) -\int_0^t \kappa_y(t, s) Y_{1}(s) \, ds
    - \int_0^t \nu_y(t, s) Y_{1}(s) \, dB_{s}\\
&- \int_0^t \kappa_{\mathbb{E}y}(t, s) \mathbb{E}Y_{1}(s) \, ds - \int_0^t \nu_{\mathbb{E}y}(t, s) \mathbb{E}Y_{1}(s)dB_{s}, X(t) \Big\rangle_{H} \Big]\, dt\\
&=\mathbb{E}\int_{0}^{b}\Big[\langle Y_{1}(t), X(t)\rangle_{H}-\int_0^t \langle \kappa_y(t, s) Y_{1}(s) , X(t)\rangle_{H}\, ds-\int_0^t \langle \nu_y(t, s) Y_{1}(s) , \aleph(t,s)\rangle_{L^0_{2}}\, ds\\
&-\int_0^t \langle \kappa_{\mathbb{E}y}(t, s) \mathbb{E}Y_{1}(s) , X(t)\rangle_{H}\, ds-\int_0^t \langle \nu_{\mathbb{E}y}(t, s)\mathbb{E} Y_{1}(s) , \aleph(t,s)\rangle_{L^0_{2}}\, ds\Big]dt\\
&=\mathbb{E}\int_{0}^{b}\Big\langle Y_{1}(t), X(t)\Big\rangle_{H}dt-\mathbb{E}\int_{0}^{b}\Big\langle Y_{1}(t), \int_{t}^{b}\kappa_{y}(s,t)^{*}X(s)ds \Big\rangle_{H}dt\\
&-\mathbb{E}\int_{0}^{b}\Big\langle Y_{1}(t), \int_{t}^{b}\nu_{y}(s,t)^{*}\aleph(s,t)ds \Big\rangle_{H}dt-\mathbb{E}\int_{0}^{b}\Big\langle \mathbb{E}Y_{1}(t), \int_{t}^{b}\kappa_{\mathbb{E}y}(s,t)^{*}X(s)ds \Big\rangle_{H}dt\\
&-\mathbb{E}\int_{0}^{b}\Big\langle \mathbb{E}Y_{1}(t), \int_{t}^{b}\nu_{\mathbb{E}y}(s,t)^{*}\aleph(s,t)ds \Big\rangle_{H}dt=\mathbb{E}\int_{0}^{b}\Big\langle Y_{1}(t), X(t)\Big\rangle_{H}dt\\
&-\mathbb{E}\int_{0}^{b}\Big\langle Y_{1}(t), \int_{t}^{b}\kappa_{y}(s,t)^{*}X(s)ds \Big\rangle_{H}dt-\mathbb{E}\int_{0}^{b}\Big\langle Y_{1}(t), \int_{t}^{b}\nu_{y}(s,t)^{*}\aleph(s,t)ds \Big\rangle_{H}dt\\
&-\mathbb{E}\int_{0}^{b}\Big\langle Y_{1}(t), \mathbb{E}\int_{t}^{b}\kappa_{\mathbb{E}y}(s,t)^{*}X(s)ds \Big\rangle_{H}dt-\mathbb{E}\int_{0}^{b}\Big\langle Y_{1}(t), \mathbb{E}\int_{t}^{b}\nu_{\mathbb{E}y}(s,t)^{*}\aleph(s,t)ds \Big\rangle_{H}dt\\
&=\mathbb{E}\int_{0}^{b}\Big\langle Y_{1}(t), X(t)-\int_{t}^{b}\kappa_{y}(s,t)^{*}X(s)ds-\int_{t}^{b}\nu_{y}(s,t)^{*}\aleph(s,t)ds\\
&-\mathbb{E}\int_{t}^{b}\kappa_{\mathbb{E}y}(s,t)^{*}X(s)ds-\mathbb{E}\int_{t}^{b}\nu_{\mathbb{E}y}(s,t)^{*}\aleph(s,t)ds\Big\rangle_{H}dt\\
&=\mathbb{E}\int_{0}^{b}\Big\langle Y_{1}(t), g_{y}(t)+\mathbb{E}g_{\mathbb{E}y}(t)-\int_{t}^{b}\aleph(t,s)dB_{s}\Big\rangle_{H}dt\\
&=\mathbb{E}\int_{0}^{b}\Big\langle Y_{1}(t), g_{y}(t)\Big\rangle_{H}dt+\mathbb{E}\int_{0}^{b}\Big\langle Y_{1}(t), \mathbb{E} g_{\mathbb{E}y}(t)\Big\rangle_{H}dt\\
&=\mathbb{E}\int_{0}^{b}\Big\langle Y_{1}(t), g_{y}(t)\Big\rangle_{H}dt+\mathbb{E}\int_{0}^{b}\Big\langle \mathbb{E}Y_{1}(t), g_{\mathbb{E}y}(t)\Big\rangle_{H}dt.
\end{aligned}\]

The equality above illustrates the duality principle between equations \eqref{1} and \eqref{3}. The initial equality stems directly from equation \eqref{3}. By applying the definitions of the M-solution and the adjoint operator, along with Fubini’s theorem, we derive the second and third equalities. The fifth equality is obtained directly from equation \eqref{1}. Finally, the properties of stochastic integrals validate the last equality. Using the definition of \(\bar{\varphi}(\cdot)\), we can then conclude the result.

\begin{align*}
I_{1}+I_{2}&=\mathbb{E}\int_{0}^{b}\Big\langle Y_{1}(t), g_{y}(t)\Big\rangle_{H}dt+\mathbb{E}\int_{0}^{b}\Big\langle \mathbb{E}Y_{1}(t), g_{\mathbb{E}y}(t)\Big\rangle_{H}dt=\mathbb{E}\int_{0}^{b}\langle \bar{\varphi}(t), X(t)\rangle_{H} dt\\
&=\mathbb{E}\int_{0}^{b}\Big\langle \int_{t}^{b}\kappa_{u}(s,t)^{*}X(s)ds+\int_{t}^{b}\nu_{u}(s,t)^{*}\aleph(s,t)ds, v(t)-\bar{u}(t)\Big\rangle_{U}dt.
\end{align*}
As a result, we get
\begin{align*}
    0\leq I_{1}+I_{2}+I_{3}=\mathbb{E}\int_{0}^{b}\Big\langle g_{u}(t)+\int_{t}^{b}\kappa_{u}(s,t)^{*}X(s)ds+\int_{t}^{b}\nu_{u}(s,t)^{*}\aleph(s,t)ds, v(t)-\bar{u}(t)\Big\rangle_{U}dt
\end{align*}
 which implies \eqref{2} since $v(\cdot)$ is arbitrary.
\end{proof}

\section{Conclusion}

This paper provides a comprehensive analysis of the well-posedness of singular MF-BSVIEs in infinite-dimensional spaces. By establishing key lemmas and conditions for the existence and uniqueness of adapted M-solutions, we offer a solid framework for understanding these complex equations.  
\bigskip  

The main findings of this study include:  
\bigskip  

$\bullet$ The existence and uniqueness of adapted M-solutions for MF-BSVIEs, derived from suitable regularity conditions on the coefficients.  
\bigskip  

$\bullet$ An extension of the theoretical analysis to singular MF-FSVIEs, demonstrating the existence of unique adapted solutions.  
\bigskip  

$\bullet$ The development of fundamental lemmas that offer the necessary tools for solving these equations in infinite-dimensional spaces.  
\bigskip  

$\bullet$ Applications to control theory, where the results can be applied to establish stochastic maximum principles.  
\bigskip  

The framework introduced in this paper opens new possibilities for future research in mean-field stochastic equations, particularly in control systems and financial models. Future work will focus on more specific applications and the development of computational methods for solving these equations in real-world contexts. In particular, the regularity of equation \eqref{r0} remains an open problem for future exploration.  

Additionally, future research will extend to the well-posedness of mean-field backward doubly stochastic singular Volterra integral equations. These equations, characterized by their interplay of backward stochastic dynamics, doubly stochastic components, and singular kernels, present unique analytical and computational challenges. Investigating their regularity, existence, and uniqueness properties will provide deeper insights into the behavior of such systems in infinite-dimensional spaces.  

The regularity of singular MF-BSVIEs in infinite-dimensional spaces is especially important due to the complexities associated with the infinite-dimensional environment. Unlike finite-dimensional spaces, where regularity can typically be established using well-known methods, infinite-dimensional spaces introduce unique challenges due to the interaction between stochastic processes and the underlying functional spaces. The singular nature of these equations, often characterized by singular kernels or degenerate behaviors, further complicates the analysis.  

Investigating the regularity of these systems is crucial for understanding the well-posedness of solutions, their stability, and their relevance in real-world applications like control theory and financial modeling, where the dynamics are inherently nonlinear and complex. Gaining insights into this regularity could also lead to more efficient numerical methods for solving these equations in practical scenarios. Future research into MF-BDSVIEs will likely reveal new methods for addressing these challenges and broaden the applications of mean-field stochastic equations.
	

\begin{thebibliography}{99}
	%

\bibitem{Ahmadova} Ahmadova, A., \& Mahmudov, N. I. (2023). Stochastic maximum principle for discrete time mean‐field optimal control problems. Optimal Control Applications and Methods, 44(6), 3361-3378.

\bibitem{Ahmadova2} Ahmadova, A., \& Mahmudov, N. I. (2024). Picard approximation of a singular backward stochastic nonlinear Volterra integral equation. Qualitative Theory of Dynamical Systems, 23(4), 192.

\bibitem{Agram} Agram, N., Hu, Y., \& Øksendal, B. (2022). Mean-field backward stochastic differential equations and applications. Systems \& Control Letters, 162, 105196.

\bibitem{Anh} Anh, V. V., Grecksch, W., \& Yong, J. (2010). Regularity of backward stochastic Volterra integral equations in Hilbert spaces. Stochastic Analysis and Applications, 29(1), 146-168.

\bibitem{Andersson} Andersson, D., \& Djehiche, B. (2011). A maximum principle for SDEs of mean-field type. Applied Mathematics \& Optimization, 63, 341-356.

\bibitem{Bismut} Bismut, J. M. (1973). Conjugate convex functions in optimal stochastic control. Journal of Mathematical Analysis and Applications, 44(2), 384-404.

\bibitem{Buckdahn} Buckdahn, R., Djehiche, B., Li, J., \& Peng, S. (2009). Mean-field backward stochastic differential equations: a limit approach.

\bibitem{Buckdahn2} Buckdahn, R., Li, J., \& Peng, S. (2009). Mean-field backward stochastic differential equations and related partial differential equations. Stochastic Processes and their Applications, 119(10), 3133-3154.

\bibitem{Gilding} Gilding, B. H. (1993). A singular nonlinear Volterra integral equation. The Journal of Integral Equations and Applications, 465-502.

\bibitem{Hamaguchi} Hamaguchi, Y. (2023). Variation of constants formulae for forward and backward stochastic Volterra integral equations. Journal of Differential Equations, 343, 332-389.

\bibitem{Hamaguchi2} Hamaguchi, Y. (2021). Infinite horizon backward stochastic Volterra integral equations and discounted control problems. ESAIM: Control, Optimisation and Calculus of Variations, 27, 101.

\bibitem{Jaber} Jaber, E. A., Neuman, E., \& Voss, M. (2023). Equilibrium in functional stochastic games with mean-field interaction. arXiv preprint arXiv:2306.05433.

\bibitem{Li} Li, Z., \& Luo, J. (2012). Mean-field reflected backward stochastic differential equations. Statistics \& Probability Letters, 82(11), 1961-1968.

\bibitem{Lin} Lin, P., \& Yong, J. (2020). Controlled singular Volterra integral equations and Pontryagin maximum principle. SIAM Journal on Control and Optimization, 58(1), 136-164.

\bibitem{Ma} Ma, J., \& Yong, J. (1999). Forward-backward stochastic differential equations and their applications (No. 1702). Springer Science \& Business Media.

\bibitem{Mahmudov} Gasimov, J. J., Asadzade, J. A., \& Mahmudov, N. I. (2024). Pontryagin maximum principle for fractional delay differential equations and controlled weakly singular Volterra delay integral equations. Qualitative Theory of Dynamical Systems, 23(5), 213.

\bibitem{Mahmudov2} Mahmudov, N. I., \& Ahmadova, A. (2022). Some results on backward stochastic differential equations of fractional order. Qualitative Theory of Dynamical Systems, 21(4), 129.

\bibitem{Pardoux} Pardoux, E., \& Peng, S. (1990). Adapted solution of a backward stochastic differential equation. Systems \& Control Letters, 14(1), 55-61.

\bibitem{Peng} Peng, S., in: J. Yan, S. Peng, S. Fang, L. Wu (Eds.), BSDE and Stochastic Optimizations; Topics in Stochastic Analysis, Science Press, Beijing, 1997 (Chapter 2) (in Chinese).
\bibitem{Prüss} Prüss, J. (2012). Evolutionary integral equations and applications. Springer Science \& Business Media.

\bibitem{Shi}  Shi, Y., Wang, T., \& Yong, J. (2011). Mean-field backward stochastic Volterra integral equations. arXiv preprint arXiv:1104.4725.


\bibitem{Shi1} Shi, Y. F., Wen, J. Q., \& Xiong, J. (2021). Mean-field backward stochastic differential equations driven by fractional Brownian motion. Acta Mathematica Sinica, English Series, 37(7), 1156-1170.


\bibitem{Wang} Wang, T., \& Zheng, M. (2024). Singular backward stochastic Volterra integral equations in infinite dimensional spaces. Journal of Differential Equations, 407, 1-56.

\bibitem{Wang2} Wang, T., \& Yong, J. (2019). Backward stochastic Volterra integral equations—representation of adapted solutions. Stochastic Processes and their Applications, 129(12), 4926-4964.

\bibitem{Wang3} Wang, H., Yong, J., \& Zhang, J. (2022, May). Path dependent Feynman–Kac formula for forward backward stochastic Volterra integral equations. In Annales de l'Institut Henri Poincaré (B) Probabilités et statistiques (Vol. 58, No. 2, pp. 603-638). Institut Henri Poincaré.

\bibitem{Wang4} Wang, T., \& Yong, J. (2023). Spike variations for stochastic Volterra integral equations. SIAM Journal on Control and Optimization, 61(6), 3608-3634.


\bibitem{Wu} Wu, H., \& Hu, J. (2023). Mean field backward doubly stochastic Volterra integral equations and their applications. Discrete \& Continuous Dynamical Systems-Series S, 16(5).

\bibitem{Yang} Yang, B., Wu, J., \& Guo, T. (2024). Well-posedness and regularity of mean-field backward doubly stochastic Volterra integral equations and applications to dynamic risk measures. Journal of Mathematical Analysis and Applications, 535(1), 128089.

\bibitem{Yong} Yong, J. (2008). Well-posedness and regularity of backward stochastic Volterra integral equations. Probability Theory and Related Fields, 142(1), 21-77.

\bibitem{Yong2} Yong, J. (2006). Backward stochastic Volterra integral equations and some related problems. Stochastic Processes and their Applications, 116(5), 779-795.

\bibitem{Zhu} Zhu, Q., Su, L., Liu, F., Shi, Y., Shen, Y. A., \& Wang, S. (2020). Mean-field type forward-backward doubly stochastic differential equations and related stochastic differential games. Frontiers of Mathematics in China, 15, 1307-1326.



\end{thebibliography}
\end{document}